\documentclass[11pt,letterpaper]{amsart}
\usepackage[usenames,dvipsnames]{color}
\usepackage{amsthm,amsfonts,amssymb,amsmath,amsxtra}
\usepackage[all]{xy}
\SelectTips{cm}{}
\usepackage{xr-hyper}
\usepackage[pdfpagelabels,pdftex,hidelinks]{hyperref}
\hypersetup{
  colorlinks   = true, %Colours links
  urlcolor     = blue, %Colour for external hyperlinks
  linkcolor    = blue, %Colour of internal links
  citecolor   = blue, %Colour of citations
pdftitle={},
  pdfauthor={},
  pdfsubject={},
  pdfkeywords={},
  breaklinks=true,
  bookmarksopen=true,
  bookmarksnumbered=true,
  pdfpagemode=UseOutlines,
  plainpages=false}
\usepackage{cleveref}

\usepackage{verbatim}
\usepackage{caption}

\usepackage{lineno}
\usepackage{mathrsfs}

\RequirePackage{xspace}
\RequirePackage{etoolbox}
\RequirePackage{varwidth}
\RequirePackage{enumitem}
\RequirePackage{tensor}
\RequirePackage{mathtools}
\RequirePackage{longtable}
\RequirePackage{multirow}

\def\ge{\geqslant}
\def\le{\leqslant}
\def\a{\alpha}
\def\b{\beta}
\def\g{\gamma}
\def\G{\Gamma}
\def\d{\delta}

\def\e{\epsilon}

\def\s{\sigma}

\def\th{\theta}
\def\k{\kappa}
\def\l{\lambda}
\def\z{\zeta}

\def\i{^{-1}}

\def\<{\langle}
\def\>{\rangle}

\newcommand{\fkg}{\ensuremath{\mathfrak{g}}\xspace}

\newcommand{\fkl}{\ensuremath{\mathfrak{l}}\xspace}

\newcommand{\fkp}{\ensuremath{\mathfrak{p}}\xspace}

\newcommand{\fkt}{\ensuremath{\mathfrak{t}}\xspace}
\newcommand{\fku}{\ensuremath{\mathfrak{u}}\xspace}

\newcommand{\fkF}{\ensuremath{\mathfrak{F}}\xspace}

\newcommand{\fkL}{\ensuremath{\mathfrak{L}}\xspace}

\newcommand{\BA}{\ensuremath{\mathbb {A}}\xspace}

\newcommand{\BF}{\ensuremath{\mathbb {F}}\xspace}
\newcommand{{\BG}}{\ensuremath{\mathbb {G}}\xspace}

\newcommand{{\BK}}{\ensuremath{\mathbb {K}}\xspace}

\newcommand{\BQ}{\ensuremath{\mathbb {Q}}\xspace}
\newcommand{\BR}{\ensuremath{\mathbb {R}}\xspace}

\newcommand{\BZ}{\ensuremath{\mathbb {Z}}\xspace}

\newcommand{\CE}{\ensuremath{\mathcal {E}}\xspace}
\newcommand{\CF}{\ensuremath{\mathcal {F}}\xspace}
\newcommand{\CG}{\ensuremath{\mathcal {G}}\xspace}
\newcommand{\CH}{\ensuremath{\mathcal {H}}\xspace}
\newcommand{\CI}{\ensuremath{\mathcal {I}}\xspace}

\newcommand{\CK}{\ensuremath{\mathcal {K}}\xspace}
\newcommand{\CL}{\ensuremath{\mathcal {L}}\xspace}
\newcommand{\CM}{\ensuremath{\mathcal {M}}\xspace}

\newcommand{\CO}{\ensuremath{\mathcal {O}}\xspace}
\newcommand{\CP}{\ensuremath{\mathcal {P}}\xspace}
\newcommand{\CQ}{\ensuremath{\mathcal {Q}}\xspace}
\newcommand{\CR}{\ensuremath{\mathcal {R}}\xspace}

\newcommand{\CY}{\ensuremath{\mathcal {Y}}\xspace}

\newcommand{\ad}{{\mathrm{ad}}}

\DeclareMathOperator{\Hom}{Hom}

\newcommand{\id}{\ensuremath{\mathrm{id}}\xspace}
\let\Im\relax
\DeclareMathOperator{\Im}{Im}

\newcommand{\reg}{{\mathrm{reg}}}

\DeclareMathOperator{\tr}{tr}

\newcommand{\wt}{\widetilde}
\newcommand{\wh}{\widehat}

\newcommand{\ov}{\overline}

\def\brk{{\breve k}}

\def\COk{{\CO_{\brk}}}

\def\pr{{\rm pr}}
\def\tPhi{\widetilde \Phi}

\def\aff{{\rm aff}}
\def\bx{{\mathbf x}}

\def\der{{\rm der}}
\def\ov{\overline}

\def\FT{{\rm FT}}
\def\pInd{{\rm pInd}}
\def\pRes{{\rm pRes}}
\def\Av{{\rm Av}}
\def\Infl{{\rm Infl}}
\def\For{{\rm For}}
\def\wt{\widetilde}

\def\sc{{\rm sc}}
\def\ud{\underline}
\def\unip{{\rm unip}}
\def\nilp{{\rm nilp}}
\def\log{{\rm log}}
\def\bfg{{\mathbf g}}
\def\bfp{{\mathbf p}}
\def\bfl{{\mathbf l}}
\def\bfn{{\mathbf n}}
\def\bft{{\mathbf t}}
\def\bfi{{\mathbf i}}
\newtheorem{theorem}{Theorem}
\newtheorem{proposition}[theorem]{Proposition}
\newtheorem{lemma}[theorem]{Lemma}

\newtheorem{corollary}[theorem]{Corollary}

\theoremstyle{definition}
\newtheorem{definition}[theorem]{Definition}

\numberwithin{equation}{section}
\numberwithin{theorem}{section}

\setitemize[0]{leftmargin=*,itemsep=\the\smallskipamount}
\setenumerate[0]{leftmargin=*,itemsep=\the\smallskipamount}

\renewcommand{\to}{%
   \ifbool{@display}{\longrightarrow}{\rightarrow}%
   }
\let\shortmapsto\mapsto
\renewcommand{\mapsto}{%
   \ifbool{@display}{\longmapsto}{\shortmapsto}%
   }
\newlength{\olen}
\newlength{\ulen}
\newlength{\xlen}
\newcommand{\xra}[2][]{%
   \ifbool{@display}%
      {\settowidth{\olen}{$\overset{#2}{\longrightarrow}$}%
       \settowidth{\ulen}{$\underset{#1}{\longrightarrow}$}%
       \settowidth{\xlen}{$\xrightarrow[#1]{#2}$}%
       \ifdimgreater{\olen}{\xlen}%
          {\underset{#1}{\overset{#2}{\longrightarrow}}}%
          {\ifdimgreater{\ulen}{\xlen}%
             {\underset{#1}{\overset{#2}{\longrightarrow}}}
             {\xrightarrow[#1]{#2}}}}%
      {\xrightarrow[#1]{#2}}
   }
\makeatother
\newcommand{\xyra}[2][]{%
   \settowidth{\xlen}{$\xrightarrow[#1]{#2}$}%
   \ifbool{@display}%
      {\settowidth{\olen}{$\overset{#2}{\longrightarrow}$}%
       \settowidth{\ulen}{$\underset{#1}{\longrightarrow}$}%
       \ifdimgreater{\olen}{\xlen}%
          {\mathrel{\xymatrix@M=.12ex@C=3.2ex{\ar[r]^-{#2}_-{#1} &}}}%
          {\ifdimgreater{\ulen}{\xlen}%
             {\mathrel{\xymatrix@M=.12ex@C=3.2ex{\ar[r]^-{#2}_-{#1} &}}}
             {\mathrel{\xymatrix@M=.12ex@C=\the\xlen{\ar[r]^-{#2}_-{#1} &}}}}}%
      {\mathrel{\xymatrix@M=.12ex@C=\the\xlen{\ar[r]^-{#2}_-{#1} &}}}%
   }
\makeatletter
\newcommand{\xla}[2][]{%
   \ifbool{@display}%
      {\settowidth{\olen}{$\overset{#2}{\longleftarrow}$}%
       \settowidth{\ulen}{$\underset{#1}{\longleftarrow}$}%
       \settowidth{\xlen}{$\xleftarrow[#1]{#2}$}%
       \ifdimgreater{\olen}{\xlen}%
          {\underset{#1}{\overset{#2}{\longleftarrow}}}%
          {\ifdimgreater{\ulen}{\xlen}%
             {\underset{#1}{\overset{#2}{\longleftarrow}}}
             {\xleftarrow[#1]{#2}}}}%
      {\xleftarrow[#1]{#2}}
   }
\newcommand{\isoarrow}{%
   \ifbool{@display}{\overset{\sim}{\longrightarrow}}{\xrightarrow\sim}%
   }

\newcommand{\sm}{{\,\smallsetminus\,}}

\newcommand{\bfx}{{\mathbf x}}

\usepackage{upgreek}
\makeatletter
\newcommand{\colim@}[2]{%
  \vtop{\m@th\ialign{##\cr
    \hfil$#1\operator@font lim$\hfil\cr
    \noalign{\nointerlineskip\kern1.5\ex@}#2\cr
    \noalign{\nointerlineskip\kern-\ex@}\cr}}%
}
\newcommand{\colim}{%
  \mathop{\mathpalette\colim@{\rightarrowfill@\textstyle}}\nmlimits@
}
\newcommand{\prolim@}[2]{%
  \vtop{\m@th\ialign{##\cr
    \hfil$#1\operator@font lim$\hfil\cr
    \noalign{\nointerlineskip\kern1.5\ex@}#2\cr
    \noalign{\nointerlineskip\kern-\ex@}\cr}}%
}
\newcommand{\prolim}{%
  \mathop{\mathpalette\colim@{\leftarrowfill@\textstyle}}\nmlimits@
}
\makeatother

\begin{document}

\title[Alternative construction of character sheaves on parahorics]{An alternative construction of character sheaves on parahoric subgroups}

\author{Alexander B. Ivanov}
\address{Fakult\"at f\"ur Mathematik, Ruhr-Universit\"at Bochum, D-44780 Bochum, Germany.}
\email{a.ivanov@rub.de}

\author{Sian Nie}
\address{Academy of Mathematics and Systems Science, Chinese Academy of Sciences, Beijing 100190, China}

\address{School of Mathematical Sciences, University of Chinese Academy of Sciences, Chinese Academy of Sciences, Beijing 100049, China}
\email{niesian@amss.ac.cn}

\author{Zhihang Yu}
\address{The Second Affiliated Hospital of Chongqing Medical Univesity, Chongqing 400000, China}
\address{Chongqing Medical University, Chongqing 400016, China}
\email{yuzhihang@amss.ac.cn}

\begin{abstract}
Inspired by the foundational work of Bezrukavnikov and Chan \cite{BC24} on character sheaves for parahoric subgroups and an alternative interpretation of deep level Deligne-Lusztig characters in \cite{Nie_24},  we present a parallel but closed (non-iterated) construction of character sheaves within the framework of J.--K. Yu's types. We show that our construction yields perverse sheaves, which coincide with those produced in \cite{BC24} in an iterated way. In the regular case we establish the compatibility of their Frobenius traces with deep level Deligne-Lusztig characters. As an application, we prove the positive-depth Springer's hypothesis for arbitrary characters, thereby generalizing the  generic case result of Chan and Oi \cite{CO25}. The proofs of our results make  critical use of the strategies and results from \cite{BC24} and \cite{Nie_24}. 
\end{abstract}

\maketitle

\section{Introduction}

Character sheaves form a very important tool in the representation theory of finite reductive groups of Lie type. They were introduced and studied by Lusztig in \cite{Lusztig_85I} and a couple of subsequent works. Later, Lusztig conjectured in \cite{Lusztig_17} that a similar theory should exist for reductive groups over finite local rings; he proved his conjecture for small values of $r$ and $p$ large enough. Recently, in \cite{BC24}, Bezrukavnikov and Chan confirmed this conjecture by introducing a theory of character sheaves on Moy--Prasad quotients of parahoric group schemes (a class of groups including those considered by Lusztig). In the generic case, they showed that the Frobenius traces of these character sheaves are compatible with deep level Deligne-Lusztig characters (for which we refer to \cite{Lusztig_79, CI_MPDL, Chan_siDL, CI_loopGLn, ChenS_17, ChanO_21, ChenS_23, Chan24, Nie_24, IvanovNie_25, CO25} and references therein). This compatibility is used by Chan and Oi in \cite[Theorem 9.3]{CO25} to establish the positive depth analogue of Springer's hypothesis in the same case. Other works on character sheaves in a related setting include \cite{Kim_18, Chen_21}. We also refer to \cite{NY25} for a construction of character sheaves on loop Lie algebras.

In this work, we present a parallel but closed (non-inductive) construction of character sheaves within the framework of J.--K. Yu's types \cite{Yu_01}. Our approach is inspired by two key sources: the foundational work of \cite{BC24} and an alternative interpretation of deep level Deligne--Lusztig characters proposed in \cite{Nie_24}. In the generic case, the same construction was initially introduced by Lusztig \cite{Lusztig_17}. We demonstrate that our construction can also be realized through iterated generic Iwahori-like inductions, which slightly diverge from the generic parabolic inductions developed in \cite{BC24}. Following the strategies and techniques in \cite{BC24}, we establish the $t$-exactness of the generic Iwahori-like induction and hence the perversity of our sheaves. As a consequence we establish the equivalence between our sheaves with those of \cite{BC24}. Furthermore, we prove the compatibility of their Frobenius traces with deep level Deligne--Lusztig characters in the regular case. As an application, we extend the positive-depth Springer's hypothesis to arbitrary characters, thereby generalizing the generic case result of \cite[Theorem 9.3]{CO25}. Now we explain our results in more detail.

%In the present paper, we give an inductive construction of character sheaves, which is similar to \cite{BC24}, but slightly differs from it. We will follow several ideas from \cite{Lusztig_85I} and \cite{BC24}. An advantage of our construction is that it seems slightly closer to J.-K. Yu types and also admits a closed (non-inductive) expression. Using the inductive construction, we prove that our sheaves are perverse. Moreover, we compare their Frobenius traces with deep level Deligne--Lusztig characters (for which we refer to \cite{CI_MPDL, Nie_24, CO25} and references therein). As an application we prove the positive-depth Springer's hypothesis for arbitrary characters $\phi$, generalizing the generic case of \cite[Theorem 9.3]{CO25}. Now we explain our results in more detail.

\subsection{Notation}

Let $k$ be a non-archimedean local field with residue field $\BF_q$ of characteristic $p$. Let $\brk$ be the completion of a maximal unramified extension of $k$ and let $\CO_{\brk}$ denote the ring of integers in $\brk$. Let $F$ be the Frobenius automorphism of $\brk$ over $k$. Let $G$ be a $k$-rational reductive group which splits over $\brk$. Let $T \subseteq B \subseteq G$ be a $k$-rational maximal torus of $G$ which splits over $\brk$.
Let $\CG = \CG_{{\bf x},0}$ be a parahoric model of $G$ attached to a point $\bfx$ in the apartment of $T$ over $k$ in the Bruhat--Tits building of $G$. For $0 \le r \in \widetilde \BR := \BR \sqcup \{s+; s \in \BR\}$, let $\CG_{\bx, r}$ be the $r$th Moy-Prasad subgroup of $\CG_{\bx, 0}$. For $0 \le s \le r \in \widetilde \BR$  we regard \[G_{s:r} = \CG_{\bx, s}(\CO_{\brk}) / \CG_{\bx, r+}(\CO_{\brk}),\] as an $\BF_q$-rational perfectly smooth affine group scheme. See \S\ref{sec:one-step} for more details.

\subsection{Character sheaves}
Let $\ell \neq p$ be a prime. Let $\phi \colon T(k) \to \ov \BQ_\ell^\times$ be a smooth character of depth $r\geq 0$. Suppose that $p$ is large enough, so that by \cite[Proposition 3.6.7]{Kaletha_19} $\phi$ admits a Howe factorization $(G^i,\phi_i,r_i)_{-1\leq i\leq d}$, where $G^i$ form an ascending sequence of ($\breve k$-rational) Levi subgroups of $G$ with $G^{-1} = T$ and $G^d = G$, $0 = r_{-1}<r_0<\dots<r_{d-1}\leq r_d = r$, and $\phi_i \colon G^i(k) \to \ov\BQ_\ell^\times$ is a $(G^i,G^{i+1})$-generic character of depth $r_i$, such that $\phi = \prod_{i=-1}^d\phi_i|_{T(k)}$.
%For simplicity assume $r=r_d=r_{d-1}$.
Put $s_i = r_i/2$ and
\[
\CK_{\phi,r} = G_{0:r}^r G_{s_0:r}^1 \dots G_{s_{d-1}:r}^d \quad \text{ and }\quad \CK_{\phi,r}^+ = G_{0+:r}^r G_{s_0+:r}^1 \dots G_{s_{d-1}+:r}^d.
\] We fix a $\brk$-rational Borel subgroup $T \subseteq B$ such that each $G^i \subseteq G$ is a standard Levi with respect to $B$. let $U \subseteq B$ be the unipotent radical and let $\ov U$ be the opposite of $U$. Consider the Iwahori-like subgroup \[ \CI_{\phi, U, r} = (\CK_{\phi, r} \cap U_r) T_r (\CK_{\phi, r}^+ \cap \ov U_r)\]
which also appears naturally in the Yu's construction of irreducible supercuspidal representations \cite{Yu_01} and in the study of deep level Deligne--Lusztig induction \cite{Nie_24}. The group $T_r$ admits a natural quotient $\bar T_r$ (such that $\phi$ factors through $\bar T_r^F$) and there are inclusion/projection maps $G_r \stackrel{h}{\hookleftarrow} \CI_{\phi, U, r} \stackrel{\delta}{\twoheadrightarrow} \bar T_r$. In \S\ref{sec:I-induction} we shall define a functor
\begin{equation}\label{eq:induction_intro} \pInd_{\CI_{\phi, U, r}}^{G_r} \colon D_{\bar T_r}(\bar T_r) \to D_{G_r}(G_r), 
\end{equation}
which is (essentially) $h_! \delta^\ast$ composed with an averaging functor. Let $\CL_\phi$ denote the local system on $\bar T_r$ attached to $\phi$. Then we obtain a perverse sheaf on $G_r$, attached to $\phi$:

\begin{theorem}\label{thm:intro_perverse}
There exists some (explicit) $N_\phi \in \BZ$ such that $\pInd_{\CI_{\phi, U, r}}^{G_r} \CL_\phi[N_\phi]$ is perverse. If $\phi$ is regular (i.e.,  $\CL_\phi$ has trivial stabilizer in the Weyl group of $G_r$), then it is simple perverse. 
\end{theorem}

For a precise statement, see \Cref{multi-step-perverse}.

\subsection{Iterated generic induction}
To prove \Cref{thm:intro_perverse}, we shall express $\pInd_{\CI_{\phi, U, r}}^{G_r} \CL_\phi$ through an iterated  induction procedure, following an idea of \cite{BC24}. Let $G \supseteq P = LN \supseteq T$ be a parabolic subgroup, with Levi factor $L \supseteq T$ and unipotent radical $N$. If $\ov N$ denotes the opposite of $N$, and $s =r/2$, we have the closed subgroup
\[
P_{s,r} = L_r N_{s:r} \ov N_{s+:r},
\]
of $G$. (In \cite{BC24} instead of $P_{s,r}$ the group $P_r$ was used.) There are natural maps $G_r \stackrel{i}{\hookleftarrow} P_{s,r} \stackrel{p}{\twoheadrightarrow} L_r$. 
%Following ideas of \cite{Lusztig_85I} and \cite{BC24}, in \S\ref{} 
In \S\ref{sec:one-step} we define a functor
\[\pInd_{P_{s,r}!}^{G_r} \colon D_{L_r}(L_r) \to D_{G_r}(G_r), \]
given by $i_! \circ p^\ast$ combined with an averaging functor, where $D_{-}(-)$ denote the equivariant bounded derived category of $\ell$-adic constructible sheaves. As in \cite[\S4]{BC24}, for an $(L,G)$-generic element $X_\psi \in (L_{r:r})^* \subseteq (G_{r:r})^*$, we have the generic subcategories $D_{L_r}^{\psi}(L_r) \subseteq D_{L_r}(L_r)$ and $D_{G_r}^{\psi}(G_r) \subseteq D_{G_r}(G_r)$. We then prove the following analogue of \cite[Theorem 5.16]{BC24}. (See \Cref{one-step} for a more precise statement).

\begin{theorem}\label{thm:generic_induction}
The functor $\pInd_{P_{s,r}!}^{G_r}$ restricts to a $t$-exact equivalence
\[
\pInd_{P_{s,r}!}^{G_r} \colon D_{L_r}^{\psi}(L_r) \stackrel{\sim}{\to} D_{G_r}^{\psi}(G_r).
\]
\end{theorem}

This construction can be iterated. For $0\leq i \leq d$ we may put
\[
\Psi_i \colon D_{G^{i-1}_{r_{i-1}}}(G^{i-1}_{r_{i-1}}) \to G_{G^i_{r_i}}(G^i_{r_i}), \quad \CF \mapsto  \CL_i \otimes \varepsilon_i^* \pInd_{P^i_{s_{i-1}, r_{i-1}}}^{G^i_{r_{i-1}}} \CF,
\]
where $\varepsilon \colon G_{r_i}^i \to G_{r_{i-1}}^i$ is the natural projection and $\CL_i$ is the local system on $G^i_{r_i}$ attached to $\phi_i$.

\begin{theorem}[see Theorem \ref{multi-step-perverse}]
Up to explicit shifts, $\pInd_{\CI_{\phi, U, r}}^{G_r} \CL_\phi$ is isomorphic to $\bigoplus_\CF m_\CF \Psi_d \dots \Psi_1(\CF)$,
where $\CF$ ranges over irreducible summands of $\Psi_0(\CL_{-1})$ with multiplicity $m_{\CF}$.
\end{theorem}
Note that Theorem \ref{thm:intro_perverse} follows from this combined with $\Psi_0(\CL_{-1})$ being perverse (up to a shift) by Lusztig's original work \cite{Lusztig_85I}, and the fact that a shift of $\Psi_i$ preserves perversity, which is a consequence of Theorem \ref{thm:generic_induction}.

\subsection{Trace of Frobenius}
In the case of finite groups of Lie type the Frobenius trace of a character sheaf attached to character $\th$ of a torus is related to the Deligne--Lusztig character $R_T^G(\th)$. Recall from \cite{CI_MPDL} that in our setting with Moy--Prasad quotients $G_r$ with $r\geq 0$, there are analoga of the classical Deligne--Lusztig varieties, the so called deep level Deligne--Lusztig varieties $X_r = X_{T,U,r}$ (which are essentially the preimage of $U_r$ under the Lang map $G_r \to G_r$). Attached to $\phi$ there is thus
the virtual $G_r^F$-module
\[
R_{T_r}^{G_r}(\phi) = \sum_i (-1)^i H_c^i(X_r, \ov \BQ_\ell)[\phi]
\]
which we may also regard as a (virtual) character of $G_r^F$. On the other hand, as $\CL_\phi$ comes equipped with an isomorphism $F^\ast \CL_\phi \stackrel{\sim}{\to} \CL_\phi$, the character sheaf $\pInd^{G_r}_{\CI_{\phi, U, r}} \CL_\phi[N_\phi]$ from \Cref{thm:intro_perverse} has an associated trace-of-Frobenius function $\chi_{\pInd^{G_r}_{\CI_{\phi, U, r}} (\CL_\phi)} \colon G_r^F \to \ov\BQ_\ell$. Our next main result shows that this compares nicely to the Deligne--Lusztig character:

\begin{theorem}[see \Cref{mainthm}]
Assume that $q$ is sufficiently large. Assume that $T$ is elliptic and $\phi$ is regular. Then  
\[\chi_{\pInd^{G_r}_{\CI_{\phi, U, r}} \CL_\phi[N_\phi]} = (-1)^{\dim G_r} R^{G_r}_{T_r}(\phi).\]
\end{theorem}

Note that this is similar, but more general than \cite[Theorem 10.9]{BC24}, in that we do not assume $\CL_\phi$ to be $(T,G)$-generic (i.e., $\phi$ has only one step in its Howe factorization). In the proof we again follow ideas from \cite{Lusztig_85I,BC24}.

\subsection{Application: positive depth Springer hypothesis}

In the situation of finite groups of Lie type, Springer's hypothesis expresses the restriction of a (classical) Deligne--Lusztig character $R_T^G(\theta)$ to unipotent elements (i.e., the so called Green function) as the Fourier transform of the characteristic function of the coadjoint orbit of a semisimple element of the dual Lie algebra. This was shown by Kazhdan \cite{Kaz77} and later a proof via character sheaves was given by Kazhdan--Varshavsky \cite{KV06}. A similar statement was formulated, and proven in the case of $(T,G)$-generic characters $\phi$, by Chan--Oi in \cite[Theorem 10.9]{CO25}. The proof made use of the character sheaves from \cite{BC24}. We generalize this result, removing the genericity assumption on $\phi$. The proof uses our version of character sheaves. To state the result we need some notation (see \S\ref{sec:Springer} for details). Let $\mathbf{g}_r$ denote the Lie algebra of $G_r$, and let $\mathbf{g}^\ast_{-r}$ denote its dual. Attached to a fixed character $\psi \colon k \to \ov\BQ_\ell^\times$ there is a well-behaved Wittvector-valued Fourier transform functor $T_{\psi}^{\bfg_{-r}^*} : D_c^b(\bfg^*_{-r}) \to D_c^b(\bfg_r)$, see \cite[\S9.1.3]{CO25} and \S\ref{sec:Springer}. Let $C((\bfg^*_{-r})^F)$ and $C(\bfg_{-r}^F)$ be the spaces of functions on $(\bfg^*_{-r})^F$ and $\bfg_r^F$ respectively. Let ${\rm T}_\psi^{\bfg_{-r}^*}: C((\bfg^*_{-r})^F)) \to C(\bfg_r^F)$ denotes the classical Fourier transformation of functions given by \[{\rm T}_\psi^{\bfg_{-r}^*}(f)(Y) = \sum_{X \in (\bfg^*_{-r})^F} f(X) \psi(X(Y)).\] If $\CF \in D_c^b(\bfg^*_{-r})$ is a Weil sheaf, then $T_{\psi}^{\bfg_{-r}^*} \CF$ inherits a natural Weil structure. We denote by $\chi_\CF \in C((\bfg^*_{-r})^F)$ and $\chi_{T_{\psi}^{\bfg_{-r}^*} \CF} \in C(\bfg_r^F)$ the associated functions respectively under the sheaf-function dictionary. Then we have $\chi_{T_{\psi}^{\bfg_{-r}^*} \CF}(Y) = {\rm T}_\psi^{\bfg_{-r}^*}(\chi_{\CF})(Y)$. 

Let $(\bfg_r)_{\nilp}$ denote the the set of nilpotent elements in $\bfg_r$.

%\begin{theorem}[See Theorem \ref{Sheafsprhy}]
%Let $X = \sum_{i=-1}^d X_i \in \mathbf  g^*_{-r}$. We have a canonical isomorphism\[\pInd_{\CI_r}^{G_r}(\CL_\phi) |_{(G_r)_\unip} \cong \log^* T_{\psi}^{\mathbf  g^*_{-r}}(\delta_{G_r \cdot X})|_{(G_r)_\unip}[ 2\dim \CK_{\phi, r} / T_r\CE_{\phi, r}](\frac{1}{2} \dim \CK_{\phi, r} / T_r\CE_{\phi, r} ).\]
%\end{theorem}

\begin{theorem}[see Corollary \ref{cor:Springer}]
Assume that $p, q$ are sufficiently large and $T$ is elliptic. Then there exists a regular element $X \in (\bft_{-r}^\ast)^F \subseteq (\bfg_{-r}^\ast)^F$, depending on $\phi$, such that for any $u\in (\bfg_r)^F_{\nilp}$ we have
    \[q^{\frac{1}{2} M_\phi} \cdot R^{G_r}_{T_r}(\phi)(\exp(u)) = {\rm T}_{\psi}^{\bfg^*_{-r}}(1_{G_r^F \cdot X})(u),\] where $1_{G_r^F \cdot X}$ is the characteristic function for the coadjoint $G_r^F$-orbit of $X$, $M_\phi = \sum_{i=0}^d (\dim G^i_{r_{i-1}} - \dim G^{i-1}_{r_{i-1}})$, $\exp \colon (\bfg_r)_{\nilp} \stackrel{\sim}{\to} (G_r)_{\unip}$ is the exponential map, and $\delta_{G_r \cdot X}$ is the extension-by-zero of the constant sheaf on the coadjoint orbit of $X$.
\end{theorem}

Using this theorem we establish in \Cref{cor:orbit} a relation, conjectured in \cite[Conjecture 8.4]{IvanovNie_24}, between deep level Deligne--Lusztig characters and Kirillov's orbit method (which is developed in \cite{BoyarchenkoS_08} in the relevant setting). 
%The orbit method parametrizes irreducible representations of a sufficiently nice pro-$p$-group in terms of coadjoint orbits in the dual of the Lie algebra. 

\subsection{Outline} After some preliminaries in \S\ref{sec:pre}, we study the generic Iwahori-like induction and state \Cref{thm:generic_induction} in \S\ref{sec:one-step}; its proof is given in \S\ref{sec:one-step-proof}. In \S\ref{sec:I-induction} we define the functor \eqref{eq:induction_intro} and reformulate it by iterated generic Iwahori-like inductions. In \S\ref{sec:comparison} we express \eqref{eq:induction_intro} as an intermediate extension of perverse sheaf on the very regular locus. Then in \S\ref{sec:copy} and \S\ref{sec:Frob-trace} we compute the Frobenius trace of \eqref{eq:induction_intro} and compare it with the deep level Deligne--Lusztig character. Finally, in \S\ref{sec:Springer} we prove the positive level Springer hypothesis.

\subsection*{Acknowledgments} It is clear from the context that our proofs crucially follow the strategies and methods from \cite{BC24} and \cite{CO25}. The third named author would like to thank Charlotte Chan for helpful discussions. We are indebted to George Lusztig for explaining the existence of regular characters/local systems when $q$ is sufficiently large. 

The first named author gratefully acknowledges the support of the German Research Foundation (DFG) via the Heisenberg program (grant nr. 462505253).

\section{Preliminary} \label{sec:pre}
Let $p \neq \ell$ be two different prime numbers and let $\BF_q$ a finite field of cardinality $q$ and of characteristic $p$.

\subsection{}
Let $X$ be an $\BF_q$-variety with Frobenius automorphism $\s$. We denote by $D(X)$ the bounded derived category of constructible $\ell$-adic sheaves. For a subset $Y \subseteq X$ we denote by $\d_Y \in D(X)$ the extension-by-zero sheaf of the constant sheaf on $Y$. We put $\d_y = \d_Y$ if $Y = \{y\}$ for some $y \in X$. 

Let $\CF \in D(X)$ be a complex endowed with an isomorphism $\varphi: \s^*\CF \overset \sim \to \CF$. Then we have the trace-of-Frobenius function \[\chi_{\CF, \varphi}: X(\BF_q) \to \ov\BQ_\ell, \quad x \mapsto \sum_i (-1)^i \tr(\varphi, \CH^i(\CF)_x),\] where $\CH^i(\CF)_x$ denotes the stalk at $x$ of the $i$th cohomology sheaf of $\CF$.

\subsection{}
Let $H$ be an $\BF_q$-rational algebraic group and let $X$ be an $\BF_q$-rational $H$-variety. We denote by $D_H(X)$ the $H$-equivariant derived category of constructible $\ell$-adic sheaves. Let $\mu: H \times X \to X$ denote the associated action map. In view of the diagram \[\xymatrix{ 
& {H \times X} \ar[dl]_{\pr_1} \ar[dr]^{\pr_2} \ar[r]^\mu & X \\
H & & X
}\] we define the following convolution functors \begin{align*}
    D(H) \times D(X) &\to D(X) \\ 
    (\CF, \CY) &\mapsto \CF \star_!\CY := \mu_!(\pr_1^* \CF \otimes \pr_2^* \CY); \\
    (\CF, \CY) &\mapsto \CF \star_*\CY := \mu_*(\pr_1^* \CF \otimes \pr_2^* \CY).
\end{align*}

Let $M \subseteq H$ be a closed subgroup. Let $\For_M^H: D_H(X) \to D_M(X)$ be the forgetful functor. Consider the morphism \[\iota_X: X \to H/M \times X, \quad x \mapsto (e, x).\] We have the following natural equivalences of categories \[\iota_X^! \circ \For_M^H, \ \iota_X^* \circ \For_M^H: D_H(H/M \times X) \overset \sim \to D_M(X).\] The averaging functors are defined as \begin{align*} \Av_{M!}^H: D_M(X) \overset {(\iota_X^! \circ \For_M^H)\i} \longrightarrow D_H(H/M \times X) \overset {\pr_{2 !}} \longrightarrow D_G(X); \\ \Av_{M*}^H: D_M(X) \overset {(\iota_X^* \circ \For_M^H)\i} \longrightarrow D_H(H/M \times X) \overset {\pr_{2 *}} \longrightarrow D_H(X). \end{align*} Note that $\Av_{M!}^H$ and $\Av_{M*}^H$ are left and right adjoint to $\For_M^H$ respectively.

\subsection{}
Let $X \to S$ be a vector bundle of constant rank $n \ge 1$ and let $X' \to S$ be its dual bundle. Consider the following natural diagram \[\xymatrix{
& {X \times_S X'} \ar[dl]_{\pr_1} \ar[dr]^{\pr_2} \ar[r]^\k & \BG_a \\
X & & X',
}\] where $\k: X \times_S X' \to \BG_a$ is the canonical pairing. Let $\CL$ be a non-trivial multiplicative rank one local system on $\BG_a$. The associated Fourier-Deligne transform is defined by \[\FT_\CL: D(X) \to D(X'), \quad \CF \to \pr_{2!}(\pr_1^*\CF \otimes \kappa^*\CL)[n].\]

\section{Generic Iwahori-like induction} \label{sec:one-step}
Let $k$ be a non-archimedean local field with residue field $\BF_q$. Let $\brk$ be the completion of a maximal unramified extension of $k$. Let $F$ be the Frobenius automorphism of $\brk$ over $k$. Denote by $\CO_k$ and $\COk$ the integer rings of $k$ and $\brk$ respectively. Let $\varpi$ be a fixed uniformizer of $k$.

Let $G$ be a $k$-rational reductive group which splits over $\brk$. We assume that $p \neq 2$ is not a bad prime for $G$ and $p \nmid |G_\der|$. Let $\bx$ be a point in the Bruhat-Tits building of $G$ over $k$. We denote by $\CG_{\bx, 0}$ the associated parahoric $\CO_k$-group model of $G$. For $0 \le r \in \widetilde \BR := \BR \sqcup \{s+; s \in \BR\}$ let $\CG_{\bx, r}$ be the $r$th Moy-Prasad subgroup of $\CG_{\bx, 0}$. For $0 \le s \le r \in \widetilde \BR$  we define \[G_{s:r} = \CG_{\bx, s} / \CG_{\bx, r+},\] which is an $\BF_q$-rational perfectly smooth affine group scheme.

Let $H \subseteq G$ be a closed subgroup. Following \cite[\S2.5]{CI_MPDL} one can construct an $\ov\BF_q$-rational closed subgroup $H_{s:r} \subseteq G_{s:r}$ in a similar way. We put $H_r = H_{0:r}$ for simplicity. If, moreover, $H$ is a $k$-rational subgroup, then $H_{s:r}$ is defined over $\BF_q$, and we still denote by $F$ the induced Frobenius automorphisms on $H$ or $H_{s:r}$.

\subsection{} \label{subsec:aff-root}
Let $T \subseteq G$ be a fixed $k$-rational and $\brk$-split maximal torus. For any subgroup $H$ normalized by $T$, let $\Phi_H = \Phi(H, T)$ denote the set of roots of $T$ appearing in $H$. Let $\a \in \Phi = \Phi_G$. We denote by $\a^\vee: \BG_m \to T$ its coroot and denote by $G^\a: \BG_a \to G$ a fixed parameterization of its root subgroup. We put $T^\a = \Im \a^\vee$.

Let $\tPhi = (\Phi \sqcup \{0\}) \times \BZ$ be the set of affine roots of $G$. Let $f \in \tPhi$. We write $\a_f \in \Phi \sqcup \{0\}$ and $n_f \in \BZ$ such that $f = (\a_f, n_f)$. We view $f$ as an affine function on the apartment $X_*(T)\otimes \BR$ of $T$ such that $f(x) = -\a_f(x) + n_f$. 

From now on we assume that $\bx$ lies in $X_*(T) \otimes \BR$. Let $f \in \tPhi$ such that $f(\bx) \ge 0$. Define
\begin{align*}
u_f\colon \BA^f := \BA^1 \to G_r, \quad &x \mapsto G^{\a_f}([x]\varpi^{n_f}) &\text{if $f \in \tPhi \sm \BZ$,} \\
u_f \colon \BA^f := X_*(T) \otimes \ov \BF_q \to G_r,\quad &\l \otimes x \mapsto \l(1 + [x]\varpi^{n_f}) &\text{if $f \in \BZ_{\geq 1}$,}
\end{align*}
where $\l \in X_*(T)$, $x \in \ov \BF_q$ and $[x] \in \CO_\brk$ denotes the Teichm\"{u}ler lift of $x$.

Let $B = U T$ be a Borel subgroup with unipotent radical $U$. Let $\Phi^+ = \Phi_B$, which is a positive system of $\Phi$. Associate to $\Phi^+$ a linear order $\le$ on $\tPhi$ such that for any $f, f' \in \tPhi$ we have $f < f'$ if either $f(\bx) < f(\bx)$ or $f(\bx) = f'(\bx)$ and $\a_{f'} - \a_f$ is a nontrivial sum of roots in $\Phi^+$. 

Let $r \in \BZ_{\ge 0}$. We set $\tPhi^+_r = \{f \in \tPhi; f > 0, f(\bx) \le r\}$. Consider the abelian group $\BA[r] := \prod_{f \in \tPhi_r^+} \BA^f$. Then there is an isomorphism of varieties
\begin{equation}\label{eq:u}
u \colon \BA[r] \to U_r G_{0+:r}, \quad (x_f)_f \mapsto \prod_f u_f(x_f),
\end{equation}
where the product is taken with respect to any fixed order on $\tPhi_r^+$. Let $E \subseteq \tPhi_r^+$. We define $\BA^E = \prod_{f \in E} \BA^f$ which is viewed as a subgroup of $\BA[r]$ in the natural way. Define
\[
G_r^E = u(\BA^E) \subseteq U_r G_{0+:r}.
\]
Moreover, we denote by
\[
\pr_E: U_r G_{0+:r} \cong \BA[r] \to \BA^E \cong G_r^E
\]
the natural projection. We write $\pr_E = \pr_f$ if $E = \{f\}$. Note that if $E + E, \BZ_{\ge 0} + E \subseteq E \cup \tPhi^{r+}$ with $\tPhi^{r+} = \{f \in \tPhi; r < f(\bx) \}$, then $G_r^E$ is a subgroup of $U_r G_{0+:r}$.

\subsection{} \label{subsec:parabolic}
Let $P = L N \supseteq B$ be a parabolic subgroup of $G$, where $L \supseteq T$ is the Levi subgroup and $N$ is the unipotent radical. Denote  by $\ov N$ the opposite of $N$. Let $N_T = N_L(T)$ be the normalizer of $T$ in $L$. We set $W_L = N_T / T$ and $W_{L_r} = (N_T)_r / T_r$, which are called the Weyl groups of $L$ and $L_r$ respectively.

Let $r \in \BZ_{\ge 0}$ and put $s = r/2$. Consider the following subgroups \[P_{s, r} := L_r N_{s,r} \subseteq G_r, \text{ where } N_{s,r} := N_{s:r} \ov N_{s+:r}.\] Notice that $N_{s,r}$ is a normal subgroup of $P_{s, r}$. We have the natural inclusion/quotient maps \[i: P_{s, r} \hookrightarrow G_r, \quad p: P_{s, r} \to P_{s, r} / N_{s,r} \cong L_r.\] Inspired by \cite{Lusztig_17} and \cite{BC24} we introduce the following induction and restriction functors \begin{align*}
    \pInd_{P_{s, r} !}^{G_r},\ \pInd_{P_{s, r} *}^{G_r} &: D_{L_r}(L_r) \to G_{G_r}(G_r); \\
    \pRes_{P_{s, r} !}^{G_r},\ \pRes_{P_{s, r} *}^{G_r} &: D_{G_r}(G_r) \to G_{L_r}(L_r),
\end{align*} where \begin{align*}
    \pInd_{P_{s, r} !}^{G_r} &:= \Av_{P_{s, r} !}^{G_r} \circ i_! \circ p^* \circ \Infl_{L_r}^{P_{s, r}}; \\
    \pInd_{P_{s, r} *}^{G_r} &:= \Av_{P_{s, r} *}^{G_r} \circ i_* \circ p^! \circ \Infl_{L_r}^{P_{s, r}}; \\
    \pRes_{P_{s, r} !}^{G_r} &:= p_! \circ i^* \circ \For_{L_r}^{G_r};\\ 
    \pRes_{P_{s, r} *}^{G_r} &:= p_* \circ i^! \circ \For_{L_r}^{G_r}.
\end{align*}
Note that $\pInd_{P_{s, r} !}^{G_r}$ and  $\pRes_{P_{s, r} !}^{G_r}$ are left adjoint to $\pRes_{P_{s, r} *}^{G_r}$ and $\pInd_{P_{s, r} *}^{G_r}$ respectively, see \cite[Lemma 3.6]{BC24}.

Consider the varieties
\begin{align*}
    \wh G_r &:= \{(g,h) \in G_r \times G_r; h^{-1} g h \in P_{s,r}\}, \\
    \wt G_r &:= \{(g, h P_{s, r}) \in G_r \times G_r/P_{s,r}; h^{-1} g h \in P_{s,r}\},
  \end{align*} together with the morphisms \begin{align*}
    &\eta: \wh G_r \to L_r, \quad & (g, h ) &\mapsto p(h\i g h); \\ 
    &\alpha: \wh G_r \to \wt G_r, \quad & (g,h) &\mapsto (g,hP_{s,r});\\
    &\pi: \wt G_r \to G_r, \quad & (g, h P_{s, r}) &\mapsto g.
\end{align*} For $\CF \in D_{L_r}(L_r)$ we denote by $\wt{\eta^* \CF} \in D(\wt G_r)$ the unique object such that $\a^*\wt{\eta^*\CF} \cong \eta^* \CF$. Then we have the following standard result (see \cite[Lemma 3.8]{BC24}).
\begin{lemma}\label{otherconstruct}
    For $\CF \in D_{L_r}(L_r)$ we have \[\pInd_{P_{s, r} !}^{G_r}(\CF) \cong \pi_! \wt{\eta^* \CF}[2 \dim N_{s, r}] \in D(G_r)\] by forgetting the equivariant structure on both sides.
\end{lemma}

\subsection{}
For any $r \in \BZ_{\ge 1}$ we define the following $\ov\BF_q$-linear spaces \[\fkg := G_{r:r}, \quad \fkp := P_{r:r}, \quad \fku := U_{r:r}, \quad \fkl := L_{r: r}, \quad \fkt := T_{r:r},\] whose dual spaces are denoted by $\fkg^*$, $\fkl^*$ and $\fkt^*$ respectively. Moreover,  we set $\fkt^\a := T^\a_{r:r} \subseteq \fkt$ for $\a \in \Phi$.

\begin{definition}
We say $X \in \fkl^*$ is $(L, G)$\emph{-generic} if the following two conditions hold:

($\mathfrak{ge1}$) $X |_{\fkt^\a} \neq 0$ for  $\a \in \Phi_G \sm \Phi_L$;

($\mathfrak{ge2}$) The stabilizer of $X |_\fkt$ in $W_G$ equals $W_L$.
\end{definition}

Now we fix an $(L, G)$-generic element $X_\psi \in \fkl^*$. Denote by $G_r \cdot X_\psi$ the coadjoint $G_r$-orbit of $X_\psi$. Define \[\CL_{\psi} := \FT_\CL(\d_{X_\psi}), \quad \CF_\psi :=  \FT_\CL(\d_{G_r \cdot X_\psi}).\] Let $i_\fkl: \fkl \hookrightarrow L_r$ and $i_\fkg: \fkg \hookrightarrow G_r$ be the natural inclusions. Define \[\CL_{\psi, r} := i_{\fkl *} \CL_\psi[\dim \fkl] \in D_{L_r}(L_r), \quad \CF_{\psi, r} := i_{\fkg *} \CF_\psi[\dim \fkg] \in D_{G_r}(G_r).\] Following \cite{BC24}, the associated $(L, G)$-generic subcategories are defined by \begin{align*}
        D_{L_r}^\psi(L_r) &:= \CL_{\psi, r} \star_! D_{L_r}(L_r) = \CL_{\psi, r} \star_* D_{L_r}(L_r) \\ D_{G_r}^\psi(L_r) &:= \CF_{\psi, r} \star_! D_{G_r}(G_r) = \CF_{\psi, r} \star_* D_{G_r}(G_r).
\end{align*}

Now we state the main result of this section, whose proof is given in \Cref{sec:one-step-proof}.
\begin{theorem} \label{one-step}
     Let $r \in \BZ_{\ge 0}$ and $s = r/2$. Then  $\pInd_{P_{s, r} !}^{G_r}$ restricts to a $t$-exact equivalence from $D_{L_r}^\psi(L_r)$ to $D_{G_r}^\psi(G_r)$, whose inverse is given by $\CL_{\psi, r} \star_! \pRes_{P_{s, r} !}^{G_r}$.
\end{theorem}

\section{Induction of Yu type} \label{sec:I-induction} 
Let $\phi: T^F = T(k) \to \ov \BQ_\ell^\times$ be a character of depth $r \in \BZ_{\ge 0}$. Thanks to \cite{Kaletha_19}, there exists a Howe factorization $(G^i, \phi_i, r_i)_{-1 \le i \le d}$ of $\phi$ satisfying the following conditions: \begin{itemize}
    \item $T= G^{-1} \subseteq G^0 \subsetneq G^1 \subsetneq \cdots \subsetneq G^d = G$ are $k$-rational Levi subgroups of $G$;

    \item $0 =: r_{-1} < r_0 < \cdots < r_{d-1} \le r_d$ if $d \ge 1$ and $0 \le r_0$ if $d = 0$;

    \item $\phi_i: G^i(k) \to \ov\BQ_\ell^\times$ is a  character of depth $r_i$, and trivial over $G_\der^i(k)$ \footnote{In \cite[Definition 3.6.2]{Kaletha_19}, $\phi_i$ is only required to be trivial over $G^i_\sc(k)$. However, the proof of \cite[Lemma 3.6.9]{Kaletha_19} shows that $\phi_i$ can be chosen to be trivial over $G_\der^i(k)$.} for $-1 \le i \le d$.

    \item $\phi_i$ is of depth $r_i$ and is $(G^i, G^{i+1})$-generic in the sense of \cite[\S 9]{Yu_01} for $0 \le i \le d-1$;

    \item $\phi = \prod_{i=-1}^d \phi_i |_{T(k)}$.
\end{itemize} 
Choose a Borel subgroup $B = T U$ with unipotent radical $U$ such that each $G^i$ is a standard levi subgroup with respect to $B$. Let \[P^i := G^{i-1} (U \cap G^i) = G^{i-1} N^i\] be the parabolic subgroup of $G^i$ with Levi subgroup $G^{i-1}$ and unipotent radical $N^i$. Let $s_i = r_i/2$ for $0 \le i \le d$. We define \[P_{s_{i-1}, r}^i = \varepsilon_i\i(P_{s_{i-1}, r_{i-1}}^i),\] where $\varepsilon_i: G_r^i \to G_{r_{i-1}}^i$ is the natural projection and $P^i_{s_{i-1}, r_{i-1}}$ is defined as in \Cref{subsec:parabolic} by taking $P = P^i$, $r =r_i$ and $s = s_{i-1}$.

Let $\ov U$ be the opposite of $U$ and set $T^i_\der = G^i_\der \cap T$. We consider the following subgroups. \begin{align*}
    \CK_{\phi, r} &= G^0_{0:r} G^1_{s_0:r} \cdots G^d_{s_{d-1}:r}; \\
    \CK_{\phi, r}^+ &= G^0_{0+:r} G^1_{s_0+:r} \cdots G^d_{s_{d-1}+:r}; \\
    T_{\phi, r} & = (T^0_\der)_{0+:r} (T^1_\der)_{r_0+:r} \cdots (T^d_\der)_{r_{d-1}+:r}; \\
    \CE_{\phi, r} &= (\CK_{\phi, r}^+ \cap U_r) T_{\phi, r} (\CK_{\phi, r}^+ \cap \ov U_r); \\
    \CI_{\phi, U, r}^\dag & = (\CK_{\phi, r} \cap U_r) T_{\phi, r} (\CK_{\phi, r}^+ \cap \ov U_r); \\
    \CI_{\phi, U, r} &= T_r \CI_{\phi, U, r}^\dag = (\CK_{\phi, r} \cap U_r) T_r (\CK_{\phi, r}^+ \cap \ov U_r); \\
    \CQ_{\phi, r}^i &= (G_\der^i)_{r_i+:r}(T_\der^{i+1})_{r_i+:r} \cdots (T_\der^d)_{r_{d-1}+:r}, \ -1 \le i \le d.    
\end{align*} 
We define $\bar G^i_r = G_r^i / \CQ^i_{\phi, r}$ and let $\bar P^i_{s_{i-1}, r}$, $\bar T^i_r$, $\bar \CI^i_r$, $\hat G^{i-1}_r$ be the natural images of $P_{s_{i-1}, r}^i$, $T^i_r$, $\CI_{\phi, U, r} \cap G^i_r$ and $G^{i-1}_r$ in $\bar G^i_r$ respectively. We put \[\bar T_r = \bar T^{-1}_r = T_r/ \CQ_{\phi, r}^{-1} = T_r / T_{\phi, r}.\] Noticing that $\phi|_{T_r(\BF_q)}$ factors through $\bar T_r(\BF_q) = T_r(\BF_q) / T_{\phi, r}(\BF_q)$, we denote by $\CL_\phi$ the associated rank one multiplicative local system on $\bar T_r$.

\subsection{} Let $p_i: \bar P_{s_{i-1}, r}^i \to \bar G^{i-1}_r$ and $\d_i: \bar \CI^i_r \to \bar T_r$ be the natural projection maps. Let $j_i: \bar P_{s_{i-1}, r}^i \hookrightarrow \bar G^i_r$ and $h_i: \bar \CI^i_r \hookrightarrow \bar G^i_r$ be the inclusion maps.

%Note that the kernels of the natural projections $\bar P^i_{s_{i-1}, r} \to \bar G^{i-1}_r$, $\hat G^{i-1}_r \to \bar G^{i-1}_r$ and $\bar T^i_r \to \bar T_r$ are connected unipotent subgroups. Hence the associated inflation functors are equivalences of categories. 
 
%\begin{align*}\Infl_{\bar G^{i-1}_r}^{\bar P^i_{s_{i-1}, r}} &: D_{\bar G^{i-1}_r}(\bar G^{i-1}_r) \to D_{\bar P^i_{s_{i-1}, r}}(\bar G^{i-1}_r) \\\Infl_{\bar G^{i-1}_r}^{\hat G^{i-1}_r} &: D_{\bar G^{i-1}_r}(\bar G^{i-1}_r) \to D_{\hat G^{i-1}_r}(\bar G^{i-1}_r).\end{align*} 

Now we introduce the following induction/restriction functors \begin{align*}
    \pInd_{\bar P^i_{s_{i-1}, r}}^{\bar G^i_r} = \Av_{\bar P^i_{s_{i-1}, r} !}^{\bar G^i_r} \circ j_{i !} \circ p_i^* \circ \Infl_{\bar G^{i-1}_r}^{\bar P^i_{s_{i-1},r}} &: D_{\bar G^{i-1}_r}(\bar G^{i-1}_r) \to D_{\bar G^i_r}(\bar G^i_r); \\ \pRes_{\bar P^i_{s_{i-1}, r}}^{\bar G^i_r} = (\Infl_{\bar G^{i-1}_r}^{\hat G^{i-1}_r})\i \circ p_{i *} \circ j_i^! \circ \For_{\hat G^{i-1}_r}^{\bar G^i_r} &: D_{\bar G^i_r}(\bar G^i_r) \to D_{\bar G^{i-1}_r}(\bar G^{i-1}_r); \\ \pInd_{\bar \CI^i_r}^{\bar G^i_r} = \Av_{\bar \CI^i_r !}^{\bar G^i_r} \circ h_{i !} \circ \d_i^* \circ \Infl_{\bar T_r}^{\bar \CI^i_r} &: D_{\bar T_r}(\bar T_r) \to D_{\bar G^i_r}(\bar G^i_r); \\ \pRes_{\bar \CI^i_r}^{\bar G^i_r} = (\Infl_{\bar T_r}^{ \bar T^i_r})\i \circ \d_{i *} \circ h_i^! \circ \For_{\bar T_r^i}^{\bar G^i_r} &: D_{\bar G^i_r}(\bar G^i_r) \to D_{\bar T_r}(\bar T_r).
\end{align*} Note that the kernels of the natural projections $\bar P^i_{s_{i-1}, r} \to \bar G^{i-1}_r$, $\hat G^{i-1}_r \to \bar G^{i-1}_r$ and $\bar T^i_r \to \bar T_r$ are connected unipotent subgroups. Hence the associated inflation functors above are equivalences of categories, and hence  invertible.

\begin{lemma} \label{adjoint}
    The functors $\pInd_{\bar P^i_{s_{i-1}, r}}^{\bar G^i_r}$ and $\pInd_{\bar \CI^i_r}^{\bar G^i_r}$ are left adjoint to $\pRes_{\bar P^i_{s_{i-1}, r}}^{\bar G^i_r}$ and $\pRes_{\bar \CI^i_r}^{\bar G^i_r}$ respectively.
\end{lemma}
\begin{proof}
    Since $\For_{\hat G^{i-1}_r}^{\bar P^i_{s_{i-1}, r}} \circ \Infl_{\bar G^{i-1}_r}^{\bar P^i_{s_{i-1}, r}} = \Infl_{\bar G^{i-1}_r}^{\hat G^{i-1}_r}$, the right adjoint of $\pInd_{\bar P^i_{s_{i-1}, r}}^{\bar G^i_r}$ is given by \begin{align*} &\quad\ (\Infl_{\bar G^{i-1}_r}^{\bar P^i_{s_{i-1}, r}})\i \circ p_{i *} \circ j_i^! \circ \For_{\bar P^i_{s_{i-1}, r}}^{\bar G^i_r} \\
    &= (\Infl_{\bar G^{i-1}_r}^{\hat G^{i-1}_r})\i \circ \For_{\hat G^{i-1}_r}^{\bar P^i_{s_{i-1}, r}} \circ p_{i *} \circ j_i^! \circ \For_{\bar P^i_{s_{i-1}, r}}^{\bar G^i_r} \\ &= (\Infl_{\bar G^{i-1}_r}^{\hat G^{i-1}_r})\i \circ p_{i *} \circ j_i^! \circ \For_{\hat G^{i-1}_r}^{\bar G^i_r} \\ &=\pRes_{\bar P^i_{s_{i-1}, r}}^{\bar G^i_r}. 
    \end{align*} As $\For_{\bar T_r}^{\bar T_r^i} \circ \Infl_{\bar T_r}^{\bar \CI_r^i} = \Infl_{\bar T_r}^{\bar T_r^i}$, it follows in the same way that $\pInd_{\bar I^i_r}^{\bar G^i_r}$ is left adjoint to $\pRes_{\bar \CI^i_r}^{\bar G^i_r}$.   
\end{proof}

\begin{proposition} \label{iteration}
    We have \begin{align*} \pRes_{\bar \CI^{i+1}_r}^{\bar G^{i+1}_r} &= \pRes_{\bar \CI^i_r}^{\bar G^i_r} \circ \pRes_{\bar P^{i+1}_{s_i, r}}^{\bar G^{i+1}_r} \\
    \pInd_{\bar \CI^{i+1}_r}^{\bar G^{i+1}_r} &= \pInd_{\bar P^{i+1}_{s_i, r}}^{\bar G^{i+1}_r} \circ \pInd_{\bar \CI^i_r}^{\bar G^i_r}.
    \end{align*}
\end{proposition}
\begin{proof}
    By the adjointness in \Cref{adjoint}, it suffices to show the first equality. First note that \[\For_{\bar T^{i+1}_r}^{\hat G^i_r} \circ \Infl_{\bar G^i_r}^{\hat G^i_r} = \Infl_{\bar T^i_r}^{\bar T^{i+1}_r} \circ \For_{\bar T^i_r}^{\bar G^i_r}.\] Consider the following diagram \[\xymatrix{
    \bar T & \bar \CI^i_r \ar[l]^{\d_i}  \ar[d]^{h_i} &  \bar \CI^{i+1}_r \ar[l]^{p_{i+1}'} \ar[d]_{h_i'} & \\ 
    &  \bar G^i_r & \bar P^{i+1}_{s_i, r} \ar[r]_{j_{i+1}} \ar[l]^{p_{i+1}} & \bar G^{i+1}_r,
    }\] where the square is Cartesian. In particular, $\d_{i+1} = \d_i \circ p_{i+1}'$ and $h_{i+1} = j_{i+1} \circ h_i'$. Therefore, \begin{align*}
        &\quad\ \pRes_{\bar \CI^i_r}^{\bar G^i_r} \circ \pRes_{\bar P^{i+1}_{s_i, r}}^{\bar G^{i+1}_r} \\
        &= (\Infl_{\bar T_r}^{ \bar T^i_r})\i \circ \d_{i *} \circ h_i^! \circ \For_{\bar T^i_r}^{\bar G^i_r} \circ (\Infl_{\bar G^i_r}^{\hat G^i_r})\i \circ p_{i+1 *} \circ j_{i+1}^! \circ \For_{\hat G^i_r}^{\bar G^{i+1}_r} \\
        &= (\Infl_{\bar T_r}^{ \bar T^i_r})\i \circ \d_{i *} \circ h_i^! \circ (\Infl_{\bar T^i_r}^{\bar T^{i+1}_r})\i \circ \For_{\bar T^{i+1}_r}^{\hat G^i_r} \circ p_{i+1 *} \circ j_{i+1}^! \circ \For_{\hat G^i_r}^{\bar G^{i+1}_r} \\
        &= (\Infl_{\bar T_r}^{ \bar T^{i+1}_r})\i \circ \d_{i *} \circ h_i^! \circ p_{i+1 *} \circ j_{i+1}^! \circ \For_{\bar T^{i+1}_r}^{\bar G^{i+1}_r} \\
        &= (\Infl_{\bar T_r}^{ \bar T^{i+1}_r})\i \circ \d_{i *} \circ p'_{i+1, *} \circ (h_i')^! \circ j_{i+1}^! \circ \For_{\bar T^{i+1}_r}^{\bar G^{i+1}_r} \\
        &= (\Infl_{\bar T_r}^{ \bar T^{i+1}_r})\i \circ \d_{i+1 *} \circ h_{i+1}^! \circ \For_{\bar T^{i+1}_r}^{\bar G^{i+1}_r} \\
        &= \pRes_{\bar \CI^{i+1}_r}^{\bar G^{i+1}_r},
    \end{align*} where the second equality follows from that $\For_{\bar T^{i+1}_r}^{\hat G^i_r} \circ \Infl_{\bar G^i_r}^{\hat G^i_r} = \Infl_{\bar T^i_r}^{\bar T^{i+1}_r} \circ \For_{\bar T^i_r}^{\bar G^i_r}$ and fourth one follows by the proper base change theorem. The proof is finished.  
\end{proof}

\subsection{} \label{subsec:multi-step}
Notice that each character $\phi_i$ is of depth $r_i$ and trivial over $G^i_\der(k)$. It restricts to a character of $G_{r_i}^i(\BF_q) / (G_\der^i)_{r_i}(\BF_q) = (G_{r_i} / G_\der^i)(\BF_q)$. Let $\CL_i$ be the corresponding multiplicative rank one $F$-equivariant local system on $\CL_i$ on $G^i_{r_i} / (G^i_\der)_{r_i}$. Consider the following natural quotient maps \begin{align*}
    q_i &: \bar G_r^i \to G^i_{r_i}; \\
    \varepsilon_i &: G^i_{r_i} \to G^i_{r_{i-1}}; \\
    \e_i &: G^i_{r_i} \to G^i_{r_i} / (G^i_\der)_{r_i}; \\
    q_{i, t} &: \bar G^i_r \to G^t_{r_t} / (G^t_\der)_{r_t}, ~ i \le t.
\end{align*}
We put $\CL^i = \bigotimes_{t \ge i+1} q_{i, t}^* \CL_t \in D_{\bar G_r^i}(\bar G_r^i)$. Note that $\CL_\phi = \bigotimes_{-1 \le t \le d} q_{-1, t}^* \CL_t$.

Let $\iota_i: \bar G^i_r \to \bar G^i_r/\bar P^i_{s_{i-1, r}} \times  \bar G^i_r$ be given by $x \mapsto (e, x)$. Then the pull-back functor $\iota_i^*: D_{\bar G^i_r}(\bar G^i_r/\bar P^i_{s_{i-1, r}} \times  \bar G^i_r) \to D_{\bar P^i_{s_{i-1, r}}}(\bar G^i_r)$ gives an equivalence.
\begin{lemma} \label{proj}
    Let $\CP \in D_{\bar G^i_r}(\bar G^i_r)$. Then we have \[\pInd_{\bar P^{i+1}_{s_i, r}}^{\bar G^{i+1}_r} (\CL^i \otimes \CP) \cong  q_{i+1,i+1}^* \CL_{i+1} \otimes  \CL^{i+1} \otimes \pInd_{\bar P^{i+1}_{s_i, r}}^{\bar G^{i+1}_r} \CP.\]  
\end{lemma}
\begin{proof}
     Set $\CF = q_{i+1,i+1}^* \CL_{i+1} \otimes  \CL^{i+1} \in D_{\bar G_r^{i+1}}(\bar G_r^{i+1})$. For $t \ge i+1$ we have the commutative diagram \[\xymatrix{
    \bar  P^{i+1}_r \ar[r]_{j_{i+1}} \ar[d]_{p_{i+1}} & \bar G^{i+1}_r \ar[d]_{q_{i+1, t}} \\
    \bar G^i_r \ar[r]_{q_{i, t}}  & G^t_{r_t} / (G^t_\der)_{r_t},
    }\] which means that \[p_{i+1}^* \CL^i \cong j_{i+1}^*\CF.\] Consider the morphisms \[\bar G^i_r \overset {p_{i+1}} \leftarrow \bar P^{i+1}_{s_i, r} \overset {j_{i+1}} \to \bar G^{i+1}_r \overset {\iota_{i+1}} \to \bar G^{i+1}_r/\bar P^{i+1}_{s_i, r} \times  \bar G^{i+1}_r  \overset {\pr_2} \to \bar G^{i+1}_r.\] We have \begin{align*}
        &\quad\ \pInd_{\bar P^{i+1}_{s_i, r}}^{\bar G^{i+1}_r} (\CL^i \otimes \CP) \\
        &\cong \pr_{2 !} (\iota_{i+1}^*)\i j_{i+1 !} p_{i+1}^*(\CL^i \otimes \CP) \\
        &\cong \pr_{2 !} (\iota_{i+1}^*)\i j_{i+1 !} ( j_{i+1}^* \CF \otimes p_{i+1}^*\CP) \\
        &\cong  \pr_{2 !} (\iota_{i+1}^*)\i (\CF \otimes (j_{i+1 !} p_{i+1}^* \CP) \\ 
        &\cong \pr_{2 !} ( \pr_2^* \CF \otimes (\iota_{i+1}^*)\i j_{i+1 !} p_{i+1}^* \CP) \\ 
        &\cong \CF \otimes \pr_{2 !} (\iota_{i+1}^*)\i j_{i+1 !} p_{i+1}^*\CP \\
        &\cong \CF \otimes \pInd_{\bar P^{i+1}_{s_i, r}}^{\bar G^{i+1}_r} \CP.
    \end{align*} The proof is finished.  
\end{proof}

\begin{lemma} \label{q-ind}
    We have \[\pInd_{\bar P^{i+1}_{s_i, r}}^{\bar G^{i+1}_r} \circ q_i^* \cong  q_{i+1}^* \circ \varepsilon_{i+1}^* \circ \pInd_{P^{i+1}_{s_i, r_i}}^{G^{i+1}_{r_i}}.\]
\end{lemma}
\begin{proof}
    Consider the following natural commutative diagram \[\xymatrix{
    \bar G^i_r \ar[d]_{q_i} & \bar P^{i+1}_{s_i, r}  \ar[l] \ar[r]  \ar[d] & \bar G^{i+1}_r \ar[r]  \ar[d] & \bar G^{i+1}_r/\bar P^{i+1}_{s_i, r} \times  \bar G^{i+1}_r \ar[r]  \ar[d] & \bar G^{i+1}_r \ar[d]_{\varepsilon_{i+1} \circ q_{i+1}} \\
    G^i_{r_i}  &  P^{i+1}_{s_i, r_i} \ar[r]  \ar[l]  & G^{i+1}_{r_i} \ar[r]   &  G^{i+1}_{r_i}/ P^{i+1}_{s_i, r_i} \times  G^{i+1}_{r_i} \ar[r]  & G^{i+1}_{r_i}.
    }\] The statement follows by noticing that right three squares are Cartesian.    
\end{proof}

For $0 \le i \le d$ we define \[\Psi_i: D_{G^{i-1}_{r_{i-1}}}(G^{i-1}_{r_{i-1}}) \to G_{G^i_{r_i}}(G^i_{r_i}), \quad \CF \mapsto  \e_i^* \CL_i \otimes \varepsilon_i^* \pInd_{P^i_{s_{i-1}, r_{i-1}}}^{G^i_{r_{i-1}}} \CF.\] 
\begin{proposition} \label{multi-step}
    We have \[\pInd_{\bar \CI_r^i}^{\bar G^i_r} \CL_\phi \cong \CL^i \otimes q_i^* \Psi_i \cdots \Psi_1 \Psi_0 (\CL_{-1}).\] 
\end{proposition}
\begin{proof}
    We argue by induction on $-1 \le i \le d$. If $i = -1$, the statement follows by noticing that $\CL_\phi = \CL^{-1} \otimes q_{-1}^* \CL_{-1}$. Assume the statement is true for $i \ge -1$. Let $\CF = \Psi_i \cdots \Psi_1 \Psi_0 (\CL_{-1})$. By \Cref{iteration} we have \begin{align*} &\quad\ \pInd_{\bar \CI^{i+1}_r}^{\bar G^{i+1}_r} \CL_\phi \\ 
    &\cong \pInd_{\bar P^{i+1}_{s_i, r}}^{\bar G^{i+1}_r} \circ \pInd_{\bar \CI_r^i}^{\bar G^i_r} \CL_\phi \\  
    &\cong \pInd_{\bar P^{i+1}_{s_i, r}}^{\bar G^{i+1}_r} (\CL^i \otimes q_i^* \CF_i) \\
    &\cong \CL^{i+1} \otimes q_{i+1,i+1}^* \CL_{i+1} \otimes \pInd_{\bar P^{i+1}_{s_i, r}}^{\bar G^{i+1}_r} q_i^*\CF \\ &\cong \CL^{i+1} \otimes q_{i+1,i+1}^* \CL_{i+1} \otimes q_{i+1}^* \varepsilon_{i+1}^* \pInd_{ P^{i+1}_{s_i, r_i}}^{G^{i+1}_{r_i}} \CF \\
    &\cong \CL^{i+1} \otimes q_{i+1}^*(\e_{i+1}^* \CL_{i+1} \otimes \varepsilon_{i+1}^* \pInd_{ P^{i+1}_{s_i, r_i}}^{G^{i+1}_{r_i}} \CF) \\ 
    &\cong \CL^{i+1} \otimes q_{i+1}^* \Psi_{i+1} \CF \\
    &\cong  \CL^{i+1} \otimes q_{i+1}^* \Psi_{i+1} \cdots \Psi_1 \Psi_0(\CL_{-1}),
    \end{align*} where the third and fourth isomorphisms follow from \Cref{proj} and \Cref{q-ind} respectively. The induction is finished.
\end{proof}

We set $\Psi_0^\dag = \Psi_0[\dim T_0]$ and $\Psi_i^\dag = \Psi_i[\dim G^i_{r_i} - \dim G^i_{r_{i-1}}]$ for $1 \le i \le d$. Put $\pInd_{\CI_{\phi, U, r}}^{G_r} = \pInd_{\bar \CI^d_r}^{\bar G^d_r}$.

We say $\phi$ is regular (for $G$) if $\CL_\phi$ has trivial stabilizer in $W_{G_r}$.
\begin{theorem} \label{multi-step-perverse}
  Let $N_\phi = \dim T_0 + \sum_{i=0}^d (\dim G^i_{r_i} - \dim G^i_{r_{i-1}})$. Then \[\pInd_{\CI_{\phi, U, r}}^{G_r} \CL_\phi [N_\phi] \cong \Psi_d^\dag \cdots \Psi_0^\dag(\CL_\phi) \cong \bigoplus_\CF m_\CF \Psi_d^\dag \cdots \Psi_1^\dag(\CF),\] where $\CF$ ranges over irreducible summands of $\Psi_0^\dag (\CL_{-1}[\dim T_0])$ with multiplicity $m_\CF$. Moreover, the summands $\Psi_d^\dag \cdots \Psi_1^\dag (\CF)$ are pairwise non-isomorphic simple perverse sheaves on $G_r$.

  In particular, if $\phi$ is regular, then $\pInd_{\CI_{\phi, U, r}}^{G_r} \CL_\phi [N_\phi]$ is simple perverse.
\end{theorem}
\begin{proof}
We follow the proof of \cite[Corollary 6.7]{BC24}. First note that $\pInd_{P^0_{0,0}}^{G^0_0}$ is a classical parabolic induction functor. By \cite[\S 4.3]{Lusztig_85I}, $\Psi_0^\dag(\CL_{-1}) = \pInd_{P^0_{0,0}}^{G^0_0}\CL_{-1}[\dim(T_0)]$ is semisimple perverse. By \Cref{one-step}, the functors $\Psi_i^\dag$ for $1 \le i \le d$ are fully faithful. Hencer first statement follows from \Cref{multi-step}.

Assume $\phi$ is regular, then $\CL_{-1}$ has trivial stabilizer in $W_{G^0_0}$. Hence  $\Psi_0^\dag(\CL_{-1})$ is simple perverse by \cite[\S 4.3]{Lusztig_85I}. So the second statement follows. 

%Hence so is $\Psi_0^\dag(\CL_{-1}[\dim T])$. Let $\CF$ be one of its direct summands. By \Cref{multi-step}, it suffices to show that $\CF_i := \Psi_i^\dag \cdots \Psi_1^\dag (\CF) $ is simple perverse for $0 \le i \le d$. We argue by induction on $i$. For $i=0$ we already know that $\CF_i = \CF$ is simple perverse. Assume that $\CF_i$ is simple perverse. As $\phi_i$ is $(G^i, G^{i+1})$-generic, we have $\CF_i \in D^{\psi}_{G^i_{r_i}}(G^i_{r_i})$. By \Cref{one-step}, $\pInd_{P^{i+1}_{s_i, r_i} !}^{G^{i+1}_{r_i}} \CF_i$ is simple perverse, hence $\CF_{i+1} = \Psi_{i+1}^\dag(\CF_i)$ remains simple. Thus the induction is finished.The last statement of the theorem now follows from that $\Psi_i^\dag$ are fully faithful for $0 \le i \le d$.
\end{proof}

%\begin{corollary}\label{regind}If $\phi$ is regular, then $\pInd_{\CI_{\phi, U, r}}^{G_r} \CL_\phi [N_\phi]$ is simple perverse.\end{corollary}
%\begin{proof}Note that $\phi$ is regular implies that $\phi_{-1}$ is regular for $G^0_0$. Hence $\Psi_0^\dag(\CL_{-1})$ is a simple perverse sheaf. Then the statement follows from \Cref{multi-step-perverse}.\end{proof}

\section{A comparison result} \label{sec:comparison}
In this section, we follow \cite[\S 7]{BC24} to compare two constructions of perverse sheaves on $G_r$. One is as in \Cref{multi-step-perverse}, the other is by intermediate extension from the very regular locus of $G_r$. Let $\phi$, $r$, $(G^i, \phi_i, r_i)_{-1 \le i \le d}$ and $B = TU$ be as in \Cref{sec:I-induction}. 

First we recall the notion of very regular elements introduced in \cite[Definition 5.1]{CI_MPDL}.
\begin{definition} 
  An element $\g \in  G_{\bx, 0} := \CG_{\bx,0}(\CO_\brk)$ is called \textit{very regular} if:
  \begin{enumerate}
    \item the identity component $T_\g$ of the centralizer of $\g$ in $G$ is a maximal torus,% is regular semisimple as an element of $G$,
    \item the apartment of  $T_\g$ contains $\bx$,
    \item $\a(\g) \not \equiv 1 \mod \varpi \CO_\brk$ for all roots $\a$ of $T_\g$ in $G$.
  \end{enumerate}
  
  An element in $G_r$ is called very regular if it is the image of a very regular element of $G_{\bx,0}$.
\end{definition}

For any two $\brk$-split maxima torus $S, S'$ of $G$ whose apartments containing $\bx$, we set $N_{G_r}(S, S') = \{x \in G_r; x S x\i = S'\}$ and $W_{G_r}(S, S')= N_{G_r}(S, S') / S_r$. 
\begin{proposition}\label{CI5.5analogy}
    Let $\g\in G_r$ be a very regular element. Let $h \in G_r$ such that $h\i \g h \in \CI_{\phi, U, r}$. Then there exists a unique $x \in W_{G_r}(T, T_\g)$ such that  $h \in w \CI_{\phi, U, r}$, where $w \in  N_{G_r}(T, T_\g) $ is some/any lift of $x$.
\end{proposition}
\begin{proof}
We first observe that since $\bx$ lies in the intersection of the apartments of $T$ and $T_\g$, these two maximal tori are conjugate by an element of $G_{\bx, 0}$. Therefore, we may assume without loss of generality that $T_\g = T$ and hence $\g \in T_r$. Similarly as \cite[Lemma 5.6]{CI_MPDL}, the statement then follows from the uniqueness of Iwahori decomposition of $G_{\bx, 0}$ and the very regularity of $\g$.

%Now consider the image in $G_s$. We have $\bar{\g} \in \bar{h} B_s \bar{h}^{-1}$. By \cite[Proposition 5.5]{CI_MPDL}, it follows that $\bar{h} \in \dot{w} B_s$, and hence $h \in \dot{w} B_r G^{s+}_r$. Next, consider the image in $G_{s-1}$. We have $\bar{\g} \in \bar{h} T_{s-1} \bar{h}^{-1}$. Since $\bar{h} \in \dot{w} B_{s-1}$, we may write $\bar{h} = \dot{w} u t$ with $t \in T_{s-1}$ and $u \in U_{s-1}$. Then
%\[u^{-1} \dot{w}^{-1} \bar{\g} \dot{w} u \in T_{s-1},\] and consequently \[u^{-1} \dot{w}^{-1} \bar{\g} \dot{w} u \dot{w}^{-1} \bar{\g}^{-1} \dot{w} = 1.\]By \cite[Lemma 5.6]{CI_MPDL}, this forces $u = 1$. Therefore, $h \in w T_r G^s_r$, which completes the proof.
\end{proof}

Let $G_{r, {\rm vreg}}$ be the set of very regular elements in $G_r$ and let $j_{{\rm vreg}}: G_{r,{\rm vreg}} \hookrightarrow G_r$ be the inclusion map. Set $T_{r,\mathrm{vreg}} := T_r \cap G_{r,{\rm vreg}}$. Let
\begin{equation*}
  \widetilde G_{r,{\rm vreg}} := \{(g,hT_r) \in G_{r, {\rm vreg}} \times G_r/T_r : h^{-1} g h \in T_r\}
\end{equation*} 
and consider the maps
\begin{equation*}
\xymatrix@R=10pt@C=15pt{
    & \wt G_{r,{\rm vreg}} \ar@<2pt>[dl]_{\eta_{\rm{vreg}}} \ar@<2pt>[dr]^{\pi_{{\rm vreg}}} \\
    T_{r, {\rm vreg}} & & G_{r,{\rm vreg}}}
\end{equation*}
given by
\begin{equation*}
  \eta_{{\rm vreg}}(g,hT_r) = h^{-1} g h, \qquad \pi_{{\rm vreg}}(g,hT_r) = g.
\end{equation*}

%Let $\widetilde G_r' = \left\{ (g, h \CI_{\phi, U, r}) \in G_r \times G_r/ \CI_{\phi, U, r} : h^{-1}gh \in \CI_{\phi, U, r} \right\}$.

\begin{lemma}\label{lem:tilde Gvreg}
The map $(g, hT_r) \mapsto (g, h \CI_{\phi, U, r})$ gives an isomorphism
\[
\widetilde G_{r, {\rm vreg}} \cong \left\{ (g, h \CI_{\phi, U, r}) \in G_{r, {\rm vreg}} \times G_r/ \CI_{\phi, U, r} : h^{-1}gh \in \CI_{\phi, U, r} \right\}.
\]
Moreover, under this isomorphism, $\eta_{\mathrm{vreg}}$ and $\pi_{\mathrm{vreg}}$ correspond to the restrictions of the maps $\eta$ and $\pi$ in \ref{subsec:parabolic} respectively, and are both $W_{G_r}$-torsors.
\end{lemma}\label{itot}
\begin{proof}
This follows from \Cref{CI5.5analogy} and that the map $L_\g: \CI_{\phi, U, r} \to \CI_{\phi, U, r}/ T_r$ given by $x \mapsto \g\i x \g T_r$ is surjective for $\g \in T_{r, {\rm vreg}}$.
\end{proof}

Recall that $\CL_\phi$ is the local system on $T_r/T_{\phi, r}$ associate to $\phi$. By abusing of notation we also denote by $\CL_\phi$ its pull-back under the natural projection $T_r \to T_r/T_{\phi, r}$.
\begin{theorem} \label{ext}
Let $\CL_{\phi, \mathrm{vreg}}$ denote the restriction of $\CL_\phi$ to $T_{r, {\rm vreg}}$. Then
\begin{equation*}
\pInd_{\CI_{\phi, U, r}}^{G_r}(\CL_\phi[N_\phi]) \cong (j_{\mathrm{vreg}})_{!*} ((\pi_{\mathrm{vreg}})_! {\eta_{\mathrm{vreg}}^*} \CL_{\phi, \mathrm{vreg}}[\dim G_r]).\end{equation*}  
In particular, it is $F$-equivariant, and independent of the choice of Borel subgroup $B = TU$ as in \Cref{sec:I-induction}.
\end{theorem}
\begin{proof}
We follow the proof of \cite[Theorem 7.6]{BC24}.
By \Cref{multi-step-perverse}, $\pInd_{\CI_{\phi, U, r}}^{G_r}(\CL_\phi[N_\phi])$ is semisimple perverse  and its endomorphism algebra has dimension equal to $|\mathrm{Stab}_{W_{G_0}}(\CL_{-1})| = |\mathrm{Stab}_W(\mathcal{L}_\phi)|$. The same holds for the intermediate extension $(j_{\mathrm{vreg}})_{!*} \left((\pi_{\mathrm{vreg}})_! \eta_{\mathrm{vreg}}^* \mathcal{L}_{\mathrm{vreg}}[\dim G_r]\right)$.

Hence, to show that these two perverse sheaves are isomorphic, it suffices to prove  $j_{\rm vreg}^*\pInd_{\CI_{\phi, U, r}}^{G_r}(\CL_\phi[N_\phi]) \cong 
 \pi_{\mathrm{vreg} !}\eta_{\mathrm{vreg}}^* \CL_{\phi,\mathrm{vreg}}[\dim G_r]$, which follows from \Cref{lem:tilde Gvreg}.
\end{proof}

%\begin{corollary}Let $\th: T^F \to \ov\BQ_\ell^\times$ be a character of depth $\le r$ with Howe factorization $(G^i, \th_i, r_i)_{-1 \le i \le d}$. Then\begin{itemize} \item $\Psi(\theta)$ is independent of the choice of Borel subgroup $B$ containing $T$ \item $\Psi(\th)$ is $F$-equivariant. \end{itemize} \end{corollary}

\begin{corollary}
  The perverse sheaf $\pInd_{\CI_{\phi, U, r}}^{G_r}(\CL_\phi[N_\phi])$ is isomorphic to the perverse sheaf $\CK_{\CL_\phi}$ constructed in \cite[Corollary 6.7]{BC24}.
\end{corollary}
\begin{proof}
    It follows from \cite[Theorem 7.6]{BC24} and \Cref{ext}.
\end{proof}

\section{An alternative construction} \label{sec:copy}
Following \cite{Lu90} and \cite{BC24}, we give an alternative construction of perverse sheaf on $G_r$, using a sequence of Borel subgroups. Let $T$, $\phi$, $r$, $(G^i, \phi_i, r_i)_{-1 \le i \le d}$ and be as in \Cref{sec:I-induction}.

Let $\ud B = (B^0, B^1, \dots, B^n)$ with $B^0 = B^n$ be a sequence of Borel subgroups containing $T$. Let $U^i$ be the unipotent radical of $B^i$, whose  opposite is denoted by $\ov U^i$. Set $\CI_r^i =  \CI_{\phi, U^i, r}$ for $i \in \BZ$. Let \[\b^i: \CI_r^i = (\CK_{\phi, r} \cap U_r^i) T_r  (\CK_{\phi, r}^+ \cap \ov U_r^i) \to T_r / T_{\phi, r}\] be the natural projection.
\begin{lemma}
    The map $(u, v) \mapsto \b^i(u) \b^{i+1}(v)$ for $u \in \CI_r^i$ and $v \in \CI_r^{i+1}$ induces a morphism \[\b_{\CI_r^i \CI_r^{i+1}}: \CI_r^i \CI_r^{i+1} \to T_r/T_{\phi, r}.\] In particular, $\b_{\CI_r^i \CI_r^{i+1}}(x z y) = \b^i(x) \b_{\CI_r^i \CI_r^{i+1}}(z) \b^{i+1}(y)$ for $x \in \CI_r^i$, $y \in \CI_r^{i+1}$ and $z \in \CI_r^i \CI_r^{i+1}$.
\end{lemma}
\begin{proof}
    It suffices to show that $\b^i(w) = \b^{i+1}(w)$ for $w \in \CI_r^i \cap \CI_r^{i+1}$. This follows from that each $\b^j$ is a group homomorphism.
\end{proof}

Let $X_{\ud B} = G_r \times G_r/ \CI_r^0  \times \cdots \times G_r/ \CI_r^n$ and \[Y_{\ud B} = \{(g, h_0 \CI_r^0, \dots, h_n \CI_r^n) \in X_{\ud B}; h_i\i h_{i+1} \in \CI_r^i \CI_r^{i+1} \text{ for } 0 \le i \le n-1, h_n\i g h_0 \in \CI_r^0\}.\] Moreover, there are two natural morphisms \begin{align*}
    \pi_{\ud B}: Y_{\ud B} \to G_r, &\quad (g, h_0 \CI_r^0, \cdots, h_n \CI_r^n) \mapsto g; \\
    \eta_{\ud B}: Y_{\ud B} \to T_r/T_{\phi, r}, &\quad (g, h_0 \CI_r^0, \dots, h_n \CI_r^n) \mapsto \b_0(h_0\i h_1) \dots \b_{n-1}(h_{n-1}\i h_n) \b^0(h_n\i g h_0).
\end{align*} 

We define \[\pInd_{\ud B}(\CL_\phi) = (\pi_{\ud B})_! (\eta_{\ud B})^* \CL_\phi [2 \dim(G_r / \CI_{\phi, r}) + N_{\ud B}],\] where $N_{\ud B} = \sum_{i=0}^n \dim \CI_r^i \CI_r^{i+1} / \CI_r^{i+1}$ and $\CL_\phi$ is the multiplicative local system over $T_r/T_{\phi, r}$ corresponding to $\phi$.

The main result of this section is
\begin{theorem} \label{copy}
    We have $\pInd_{\CI_r^0}^{G_r} \CL_\phi \cong \pInd_{\ud B}(\CL_\phi)$.
\end{theorem}

\subsection{}
Let $W$ be the absolute Weyl group of $T$ in $G$.  Let $w_i \in W$ such that $B^i = {}^{w_i} B^{i-1}$.  Note that each Borel subgroup $B^i$ determines a length function $\ell_{B^i}: W \to \BZ_{\ge 0}$. 

We fix $1 \le i \le n-1$ and suppose that $w_i$ is a  simple reflection with respect to $\ell_{B^{i-1}}$. Let $\a \in \Phi(G, T)$ be the simple root in $B^{i-1}$ associated to $w_i$. Then $\{-\a\} = \Phi(B^i, T) \sm \Phi(B^{i-1}, T)$.

For $f \in \tPhi$ with $f(\bx) \ge 0$ we denote by $G_r^f = G_r^{\{f\}} \subseteq G_r$ be as in \Cref{subsec:aff-root}.

For $\a \in \Phi(G, T)$ we set $j(\a)$ to be the integer $0 \le j \le d$ such that $\a \in \Phi(G^j, T) \sm \Phi(G^{j-1}, T)$, and set $r(\a) = r_{j(\a)-1}$.
\begin{lemma} \label{product}
 If $\CI_r^{i-1} \neq \CI_r^i$, then there exists $f \in \tPhi$ such that $\a_f = \a$ and $f(\bx) = r(\a)/2$. In this case, \[\CI_r^{i-1} \CI_r^i = G_r^f \CI_r^i = \CI_r^{i-1} G^{f^\dag}_r,\] where $f^\dag = r(\a) - f \in \Phi_\aff$.
\end{lemma}

Let ${}_i \ud B = (B^0, \dots, \widehat{B^i}, \dots, B^n)$ and define
\[{}_ip: Y_{\ud B} \to X_{{}_i \ud B}, \quad (g, h_0 \CI_r^0, \cdots, h_n \CI_r^n) \mapsto (g, h_0 \CI_r^0, \dots, \widehat{h_i \CI_r^i}, \dots  h_n \CI_r^n).\]

\begin{lemma} \label{fiber}
    Let $\xi = (g, h_0 \CI_r^0, \dots, \wh{h_i \CI_r^i}, \dots, h_n \CI_r^n) \in X_{{}_i\ud B}$. Then the projection map $(g', h_0' \CI_r^0, \dots, h_n' \CI_r^n) \mapsto h_i' \CI_r^i$ induces an isomorphism \[{}_ip\i(\xi) \cong h_{i-1}(\CI_r^{i-1} \CI_r^i \cap h_{i-1}\i h_{i+1} \CI_r^{i+1} \CI_r^i) / \CI_r^i.\] 
\end{lemma}

\begin{proposition} \label{braid}
    If $-\a \in \Phi(B^{i+1}, T)$, then we have \[\pInd_{\ud B}(\CL_\phi) \cong \pInd_{{}_i\ud B}(\CL_\phi).\]    
\end{proposition}
\begin{proof}
    It suffices to show that $p_i$ induces an isomorphism $Y_{\ud B} \cong Y_{{}_i \ud B}$, $\eta_{\ud B} = \eta_{{}_i\ud B} \circ {}_ip$ and $N_{\ud B} = N_{{}_i\ud B}$. 
    
    By assumption, $G_r^{f^\dag} \subseteq \CI_r^i \cap \CI_r^{i+1}$. By \Cref{product}, we have $\CI_r^{i-1}\CI_r^i\CI_r^{i+1} = \CI_r^{i-1}\CI_r^{i+1}$ and $\Im {}_ip \subseteq Y_{{}_i \ud B}$. We claim the natural map \[\mu: \CI_r^{i-1}\CI_r^i \times^{\CI_r^i} \CI_r^i\CI_r^{i+1} \to \CI_r^{i-1}\CI_r^{i+1}, \quad (x, y) \mapsto xy\] is an isomorphism. It remains to show $\eta$ is injective. We may assume $\CI_r^{i-1} \neq \CI_r^i$ and let $f, f^\dag$ be as in \Cref{product}. Let $x, x' \in \CI_r^{i-1}\CI_r^i$ and $y, y' \in \CI_r^i\CI_r^{i+1}$ such that $x y = x' y'$. We can assume further that $y, y \in \CI_r^{i+1}$. By \Cref{product} we have \[x\i x' = y {y'}\i \in U_r^{f^\dag} \CI_r^{i-1} U_r^{f^\dag} \cap \CI_r^{i+1} \subseteq \CI_r^i,\] where the last inclusion follows from the inclusions $U_r^{f^\dag} \subseteq \CI_r^i \cap \CI_r^{i+1}$ and $\CI_r^{i-1} \cap \CI_r^{i+1} \subseteq \CI_r^i$. So the claim is proved.

    As $\mu$ is an isomorphism, we have $N_{\ud B} = N_{{}_i \ud B}$, and moreover, ${}_ip$ is an isomorphism by \Cref{fiber}. The proof is finished.    
\end{proof}

\begin{proposition} \label{quadratic}
    If $B^{i-1} = B^{i+1}$, then $\pInd_{\ud B}(\CL_\phi) \cong \pInd_{{}_i\ud B}(\CL_\phi).$
\end{proposition}
\begin{proof}
    If $\CI_r^{i-1} = \CI_r^i$ the statement is trivial. We assume otherwise and let $f, f^\dag \in \Phi_\aff$ be as in \Cref{product}. Note that $N_{\ud B} = N_{{}_i \ud B} + 2$.
    Consider the decomposition $Y_{\ud B} = Y_{\ud B}' \sqcup Y_{\ud B}''$, where \begin{align*}
        Y_{\ud B}' &= \{(g, h_0 \CI_r^0, \cdots, h_n \CI_r^n); h_{i-1}\i h_{i+1} \in \CI_r^{i-1}\}; \\
        Y_{\ud B}'' &= \{(g, h_0 \CI_r^0, \cdots, h_n \CI_r^n); h_{i-1}\i h_{i+1} \notin \CI_r^{i-1}\}
    \end{align*} Let ${}_ip', \eta', \pi'$ and ${}_ip'', \eta'', \pi''$ be the restrictions of ${}_ip, \eta_{\ud B}, \pi_{\ud B}$ to $Y_{\ud B}'$ and $Y_{\ud B}''$ respectively.

    By definition, ${}_ip'$ induces a morphism from $Y_{\ud B}'$ to $Y_{{}_i \ud B}$, which we still denoted by ${}_ip'$. By \Cref{product} and \Cref{fiber}, the fiber of any point $\xi = (g, h_0 \CI_r^0, \cdots, \widehat{h_i \CI_r^i}, \dots h_n \CI_r^n) \in Y_{{}_i \ud B}$ under ${}_ip'$ is \[\{h_{i-1} u_f(z)\CI_r^i / \CI_r^i; z \in \BG_a\} \cong \BG_a.\] Moreover, the restriction of $\eta_{\ud B} / (\eta_{{}_i \ud B} \circ {}_ip')$ on ${{}_ip'}\i(\xi)$ is \begin{align*} &\quad\ \frac{\b_{\CI_r^{i-1}\CI_r^i}(h_{i-1}\i h_{i-1} u_f(z)) \b_{\CI_r^i\CI_r^{i+1}}(u_f(z)\i h_{i-1}\i h_{i+1})}{\b_{\CI_r^{i-1} \CI_r^{i+1}}(h_{i-1}\i h_{i+1})} \\ &= \frac{\b^{i-1}(u_f(z)) \b^{i+1}(u_f(z) h_{i-1}\i h_{i+1})}{\b^{i+1}(h_{i-1}\i h_{i+1})} \\ &= \b^{i-1}(u_f(z)) \b^{i+1}(u_f(z)\i) \\ & = 1 \in T_r / T_{\phi, r},\end{align*} where the second equality follows from that $h_{i-1}\i h_{i+1}, u_f(z) \in \CI_r^{i-1} = \CI_r^{i+1}$. Therefore, by the proper base change theorem we have \[(\pi')_! (\eta')^*\CL_\phi)[2 \dim(G_r/\CI_{\phi, r}) + N_{\ud B}] \cong \pInd_{{}_i \ud B}(\CL_\phi).\]

    It remains to show that $(\pi'')_! (\eta'')^*\CL_\phi = 0$.  Let $\xi = (g, h_0 \CI_r^0, \cdots, \widehat{h_i \CI_r^i}, \dots h_n \CI_r^n) \in \Im {}_ip''$. It suffices to show $({}_ip'')_! (\eta''|_{{{}_ip''}\i(\xi)})^* \CL_\phi = 0$. By \Cref{product} we may assume that $h_{i-1}\i h_{i+1} = u_{f^\dag}(z_0)$ for some $0 \neq z_0 \in \ov\BF_q$.
    
    First we assume that $r(\a) > 0$. By \Cref{fiber} we have \[{{}_ip''}\i(\xi) = \{ h_{i-1} u_{f^\dag}(z_0) u_f(z) \CI_r^i/\CI_r^i; z \in \BG_a\} \cong \BG_a.\] Then the restriction of $\eta_{\ud B} / (\eta_{{}_i \ud B} \circ {}_ip'')$ on ${{}_ip''}\i(\xi)$ is given by \begin{align*}
        &\quad\ \frac{\b_{\CI_r^{i-1}\CI_r^i}(h_{i-1}\i h_{i-1} u_{f^\dag}(z_0) u_f(z)) \b_{\CI_r^i \CI_r^{i+1}}(u_f(z)\i u_{f^\dag}(z_0)\i h_{i-1}\i h_{i+1})}{\b_{\CI_r^{i-1} \CI_r^{i+1}}(h_{i-1}\i h_{i+1})} \\
        &= \b_{\CI_r^{i-1} \CI_r^i}(u_{f^\dag}(z_0) u_f(z)) \\
        &= \a^\vee(1 + \varpi^{r(\a)} [z_0 z]) \in T_r / T_{\phi, r}.
    \end{align*} Since $\phi_{j(\a)-1}$ is $(G^{j(\a)-1}, G^{j(\a)})$-generic, we have $(p_i'')_! (\eta''|_{{{}_ip''}\i(\xi)})^* \CL_\phi = 0$ as desired.

    Then we assume that $r(\a) = 0$, that is, $\a \in \Phi(G^0, T)$. By \Cref{fiber} we have \[{{}_ip''}\i(\xi) = \{ h_{i-1} u_{f^\dag}(z_0) u_f(z_0\i(z -1)) \CI_r^i/\CI_r^i; z \in \BG_m\} \cong \BG_m.\] Then the restriction of $f\eta_{\ud B} / (\eta_{{}_i \ud B} \circ {}_ip'')$ on ${{}_ip''}\i(\xi)$ is given by \[\b_{\CI_r^{i-1} \CI_r^i}(u_{f^\dag}(z_0) u_f(z_0\i(z-1))) = \a^\vee(z).\] As $\phi_{-1}$ is regular for $G^0$, $({}_ip'')_! (\eta''|_{{{}_ip''}\i(\xi)})^* \CL_\phi = 0$ as desired. The proof is finished.
\end{proof}

\subsection{}
Let $\ud B = (B^0, \dots, B^n)$ be a sequence of Borel subgroups contains $T$. Let $w_i \in W$ be such that $B^i = {}^{w_i} B^{i-1}$. We say $\ud B$ is saturated if for each $1 \le i \le n$ we have $\ell_{B^{i-1}}(w_i) \le 1$. Moreover, we say $\ud B$ is reduced if $\ud B$ is saturated and $\ell_{B^0}(w_i w_{i-1} \cdots w_1) > \ell_{B^0}(w_{i-1} \cdots w_1)$ for $1 \le i \le n$.

We record the following standard results on root systems.
\begin{lemma} \label{subsequence}
    A sequence $(B^0, \dots, B^n)$ is reduced if and only if \[\Phi(B^m, T) \sm \Phi(B^0, T) = \bigsqcup_{i=1}^m \Phi(B^i, T) \sm \Phi(B^{i-1}, T),\] and each set on the right hand side is of cardinality one.
    
    In particular, if $(B^0, \dots, B^n)$ is reduced, then so is $(B^j, \dots, B^{j'})$ for any $0 \le j < j' \le n$.
\end{lemma}

\begin{lemma} \label{extension}
    Let $B, B'$ be two Borel subgroups. Then there exists a reduced sequence $(B^0, \dots, B^m)$ such that $B^0 = B$ and $B^m = B'$.
\end{lemma}

\begin{proposition} \label{reduction}
    Let $\ud B = (B^0, \dots, B^n)$ with $B^0 = B^n$ be a saturated sequence. Then there exist saturated sequences  \[\ud B = \ud B(0), \ud B(1), \dots, \ud B(m) = (B^0, B^0)\] with $\ud B(i) = (B_i^0, \dots, B_i^{n_i})$ such that for each $1 \le i \le m$ one of the following cases occurs: \begin{itemize}
        \item[(1)] $\ud B(i)$ is obtained from $\ud B(i-1)$ by deleting $B_{i-1}^j$ for some $1 \le j \le n_{i-1}-1$ with $B_{i-1}^{j-1} = B_{i-1}^{j+1}$; 

        \item[(2)] $\ud B(i)$ is obtained from $\ud B(i-1)$ by deleting $B_{i-1}^j$ for some $1 \le j \le n_{i-1}-1$ with $B_{i-1}^j = B_{i-1}^{j+1}$;

        \item[(3)] $\ud B(i)$ is obtained from $\ud B(i-1)$ by replacing a reduced subsequence $(B_{i-1}^j, \dots, B_{i-1}^{j'})$ for some $1 \le j < j' \le n_{i-1}$ with another reduced subsequence $(B_i^j, \dots, B_i^{j'})$ with $B_{i-1}^j = B_i^j$ and $B_{i-1}^{j'} = B_i^{j'}$.
    \end{itemize}
\end{proposition}

\begin{proposition} \label{saturation}
    Let $\ud B = (B^0, \dots, B^n)$ with $B^0 = B^n$. Let $0 \le j \le n-1$ and let \[(B^j = B^{j, 0}, \dots, B^{j, m} = B^{j+1})\] be a reduced sequence. Let $\ud B'$ be obtained from $\ud B$ by replacing $(B^j, B^{j+1})$ with the above sequence. Then $\pInd_{\ud B} \CL_\phi \cong \pInd_{\ud B'} \CL_\phi$.
\end{proposition}
\begin{proof}
    For $0 \le i \le m-1$ let $\ud B(i)$ be obtained from $\ud B$ by replacing  $(B^j, B^{j+1})$ with $(B^{j, 0}, \cdots, B^{j, i}, B^{j, m})$. In view of \Cref{subsequence}, we can apply \Cref{braid} (by taking $(B^{i-1}, B^i, B^{i+1}) = (B^{j, i-1}, B^{j, i}, B^{j, m})$ for $1 \le i \le m-1$) to deduce that \[\pInd_{\ud B} \CL_\phi = \pInd_{\ud B(0)} \CL_\phi \cong \pInd_{\ud B(1)} \CL_\phi \cong \cdots \cong \pInd_{\ud B(m-1)} \CL_\phi = \pInd_{\ud B'} \CL_\phi.\] The proof is finished.
\end{proof}

\begin{proof}[Proof of \Cref{copy}]
By Lemma \ref{extension} and \Cref{saturation}, we may assume $\ud B = (B^0, \dots, B^n)$ with $B^0 = B^n$ is saturated. If $\ud B = (B^0, B^0)$, then $Y_{\ud B} \cong \wt G_r$ and statement follows.

By \Cref{reduction}, it suffices to show that $\pInd_{\ud B(i-1)} \CL_\phi \cong \pInd_{\ud B(i)} \CL_\phi$ for $1 \le i \le m$. If Case (1) occurs, the statement follows from \Cref{quadratic}. If Case (2) occurs, the statement follows from that ${}_jp: Y_{\ud B(i-1)} \to Y_{\ud B(i)}$ is an isomorphism. If Case (3) occurs, the statement follows from \Cref{saturation}. The proof is finished. 
\end{proof}

\section{Trace of Frobenius} \label{sec:Frob-trace}
Let $T$, $\phi$, $r$ and $B = T U$ be as in \Cref{sec:I-induction}. Let $n \in \BZ_{\ge 1}$ such that $F^n B = B$. Set $\ud B = (B, F B, \dots, F^n B)$. Let $\CI_r^\dag = \CI_{\phi, B, r}^\dag$ and $\CI_r = \CI_{\phi, B, r}$, $Y = Y_{\ud B}$, $\eta_Y = \eta_{\ud B}$, $\pi_Y = \pi_{\ud B}$ and so on be as in \Cref{sec:copy}. We define \begin{align*} F_Y: Y &\to Y, \\ (g, h_0 \CI_r^0, h_1 \CI_r^1 \dots, h_n\CI_r^n) &\mapsto (F(g), F(g\i h_n)\CI_r^0, F(h_0) \CI_r^1 \dots, F(h_{n-1})\CI_r^n).\end{align*}

\begin{lemma}
    There is a commutative diagram  \[\xymatrix{
    T_r/T_{\phi, r} \ar[d]_F & Y \ar[r]^{\pi_Y} \ar[l]_{\eta_Y}  \ar[d]_{F_Y} &  G_r \ar[d]^F \\
     T_r/T_{\phi, r}  & Y  \ar[l]_{\eta_Y} \ar[r]^{\pi_Y} & G_r.
    }\] In particular, for any $M \in D(T_r / T_{\phi, r})$ such that $F^* M \cong M$ we have \[F^* \pi_{Y!} \eta_Y^* M \cong \pi_{Y!} \eta_Y^* F^*M \cong \pi_{Y!} \eta_Y^* M.\]
\end{lemma}

For $g \in G_r$ we define $Y_g = \pi_Y\i(g)$ and  \[Z_g = \{h T_r^F(\CI_r^\dag \cap F\CI_r^\dag); h\i F(h) \in F \CI_r^\dag, F^(h)\i g h \in T_r^F(\CI_r^\dag \cap F\CI_r^\dag)\}.\]
\begin{lemma}
    For $g \in G_r^F$ the restriction map $F_Y: Y_g \to Y_g$ is the Frobenius map for some $\BF_q$-rational structure of $Y_g$.
\end{lemma}
\begin{proof}
    It follows in the same way as \cite[Lemma 9.2]{BC24}.
\end{proof}

Let $\pr_T: T_r^F (\CI_r^\dag \cap F\CI_r^\dag) \to T_r^F / T_{\phi, r}^F$ denote the natural projection.
\begin{proposition} \label{Frob-trace}
    Let $M \in D(T_r/T_{\phi, r})$ be such that $F^* M \cong M$. For $g \in G_r^F$ we have \[\chi_{\pi_{Y!} \eta_Y^* M}(g) = \sum_{h T_r^F(\CI_r^\dag \cap F\CI_r^\dag) \in Z_g} \chi_M(\pr_T(F^n(h)\i g h)).\] 
\end{proposition}
\begin{proof}
    It follows in the same way as \cite[Proposition 9.3]{BC24} by replacing the pair $(U_r, B_r)$ with $(\CI_r^\dag, \CI_r)$.
\end{proof}

Recall that $\CL_\phi$ denotes the rank one multiplicative local system on $T_r / T_{\phi, r}$ associated to $\phi$.
\begin{corollary}\label{sheafscalar}
    Assume $\phi$ is regular, then there exists a constant $c$ such that for any $g \in G_r^F$ we have \[\chi_{\pInd_{\CI_r}^{G_r}(\CL_\phi)}(g) = c \cdot \sum_{h T_r^F(\CI_r^\dag \cap F\CI_r^\dag) \in Z_g} \phi(\pr_T(F^n(h)\i g h)).\]
\end{corollary}
\begin{proof}
    By \Cref{multi-step-perverse}, $\pInd_{\CI_r}^{G_r}(\CL_\phi)$ is a simple perverse sheaf up to shift. It follows from \Cref{copy} that $\chi_{\pInd_{\CI_r}^{G_r}(\CL_\phi)}$ differs from $\chi_{\pi_{Y!} \eta_Y^* \CL_\phi}$ by a scalar, and the statement follows from \Cref{Frob-trace}.
\end{proof}

\section{Comparison with deep level Deligne-Lusztig characters}
Let notation be as in \Cref{sec:Frob-trace}. Consider the following varieties \begin{align*} X_r &= \{g \in G_r; g\i F(g) \in F U_r\} / (U_r \cap FU_r); \\ Y_r &= \{g \in G_r; g\i F(g) \in F \CI_r^\dag\} / (\CI_r^\dag \cap F \CI_r^\dag).\end{align*} which admit natural actions of $G_r^F \times T_r^F$ by left/right multiplication. Consider the following virtual $G_r^F$-modules \begin{align*} R_{T_r}^{G_r}(\phi) &= \sum_i (-1)^i H_c^i(X_r, \ov \BQ_\ell)[\phi], \\  \CR_{T_r}^{G_r}(\phi) &= \sum_i (-1)^i H_c^i(Y_r, \ov \BQ_\ell)[\phi], \end{align*} where $H_c^i(-, \ov \BQ_\ell)[\phi]$ denotes the subspace of $H_c^i(-, \ov \BQ_\ell)$ on which $T_r^F$ acts via $\phi$. The virtual $G_r^F$-module $R_{T_r}^{G_r}(\phi)$ is referred to as a Deligne-Lusztig representation.

\begin{theorem}\cite[Proposition 1.1]{Nie_24}\label{niecompare}
    Assume that $T$ is elliptic. Then $R_{T_r}^{G_r}(\phi) = \CR_{T_r}^{G_r}(\phi)$.
\end{theorem}

\begin{theorem}\cite[Theorem 1.2]{CI_MPDL}\label{veryregu}
For any character $\phi$ and any very regular element $\g \in G_r^F$,
  \begin{equation*}
    {R_{T_r}^{G_r}(\phi)}(g) = \sum_{w \in W_{G_r}(T_\g, T)^F} \phi^w(g).
  \end{equation*} Here ${}^w \phi$ is character of $T_{\g, r}^F$ such that $\phi^w(g) = \phi(w g w\i)$.
\end{theorem}

\begin{proposition}
    Let $m \in \BZ_{\ge 1}$ and $g \in G_r^F$. Then \[\tr(g \circ F^{mn}; \CR_{T_r}^{G_r}(\phi)) = \sum_{h T_r^F(\CI_r^\dag \cap F\CI_r^\dag) \in Z_g} \phi(\pr_T(F^{mn}(h)\i g h)).\] Here $Z_g$ and $n$ are as in \Cref{sec:Frob-trace}.
\end{proposition}

\begin{theorem}\cite[Theorem 10.6]{BC24}
    Assume $T$ is elliptic and $\phi$ is regular. Then $\pm R_{T_r}^{G_r}(\phi)$ is an irreducible representation of $G_r^F \times \<F^n\>$. Moreover, $F^n$ acts on $R_{T_r}^{G_r}(\phi)$ by a scalar.
\end{theorem}

\begin{corollary}\label{algscalar}
    If $T$ is elliptic and $\phi$ is regular, then there exists a constant $\lambda \neq 0$ such that for $g\in G^F_r$, we have
    $$R^{G_r}_{T_r}(\phi)(g) = \lambda \cdot \sum_{h T_r^F(\CI_r^\dag \cap F\CI_r^\dag) \in Z_g} \phi(\pr_T(F^{n}(h)\i g h))$$
\end{corollary}

\begin{theorem}\label{mainthm}
    Suppose $q$ is sufficiently large. Assume $T$ is elliptic and $\phi$ is regular. Then 
    $$\chi_{\pInd^{G_r}_{\CI_r} (\CL_\phi[N_\phi])} =(-1)^{\dim(G_r)} R^{G_r}_{T_r}(\phi).$$
\end{theorem}

\begin{proof}
By \Cref{sheafscalar}, \Cref{algscalar}, \Cref{Frob-trace} and \Cref{niecompare}, there exists a constant $\mu \in \ov\BQ_\ell$ such that
\[\chi_{\pInd_{\CI_r}^{G_r}(\CL_\phi[N_\phi])} = \mu \cdot R_{T_r}^{G_r}(\phi).\] 
As $q$ is sufficiently large, it follows from \cite[Theorem 10.9]{BC24} and \Cref{veryregu} that there exists a very regular element $\g \in G^F_r$ such that
\[
R_{T_r}^{G_r}(\phi)(\g) = \sum_{w \in W_{G_r}(T_\g, T)^F} \phi^w(\g) \neq 0.
\]
On the other hand, by \Cref{ext} we have
\begin{align*}
\chi_{\pInd_{\CI_r}^{G_r}(\CL_\phi[N_\phi])}(\g) &= \chi_{\pi_{{\rm vreg}!} \eta_{\rm vreg}^* \CL_{\phi, {\rm vreg}}[\dim G_r]}(\g) \\
&= (-1)^{\dim G_r} \cdot \sum_{w \in W_{G_r}(T_\g, T)^F} \phi^w(\g) = (-1)^{\dim G_r} R_{T_r}^{G_r}(\phi)(\g) \neq 0.
\end{align*}
Thus $\mu = (-1)^{\dim G_r}$ as desired.
\end{proof}

\section{Positive level Springer hypothesis}\label{sec:Springer}

\subsection{Wittvector-valued Fourier transform} \label{subsec:W-Fourier}
            
Let $W$ be a connected commutative unipotent group scheme over $\BF_q$.
Let $E,E'$ be two connected commutative unipotent group schemes of dimension $d$ over $\BF_q$, assume a pairing
\[\xymatrix{ & E \times E' \ar[ld]^{p_1} \ar[rr]^{\mu} \ar[rd]_{p_2} & & W \\ E & & E' }. \]
is given. Let $\psi \colon W(\BF_q) \to \ov\BQ_\ell^\times$ be a non-trivial character. Via the Lang--Steinberg map $x \mapsto F(x) - x \colon W \to W$, $\psi$ defines an \'etale sheaf $\CL_\psi$ on $W$
We can then define the Fourier transform
\begin{align*}
T_\psi = T_\psi^{E,\mu} \colon D_c^b(E,\ov\BQ_\ell) &\to D_c^b(E',\ov\BQ_\ell) \\
A &\mapsto p_{2!}(p_1^\ast A \otimes \mu^\ast\CL_\psi).
\end{align*} Note that here we do not take the shift $[d]$, as in the usual definition of Fourier transform.

\begin{proposition}\label{prop:inversion}
Let $W,E,E',\mu$ as above such that the character $\psi_x \colon E'(\BF_q) \to \ov\BQ_\ell^\times$ given by $\psi_x(y) = \psi(\mu(x,y))$ is trivial if and only if $x = 0$. Then $T_{\psi^{-1}} \circ T_{\psi} (A) \simeq A[-2d](-d)$, where $(-d)$ denotes the Tate-twist.
\end{proposition}

\begin{proof}
We can closely follow the argument of \cite[Th\'eor\`eme (1.2.2.1)]{Laumon_Fourier} (or \cite[I.5]{KW}). We have the commutative diagram (all functors are derived):
\[
\xymatrix{
E\times E \ar[dd]^{q_1} & & E\times E'\times E \ar[ll]_{\pr_{13}} \ar[ld]^{\pr_{12}} \ar[rd]_{\pr_{23}} \\
& E\times E' \ar[ld]_{p_1} \ar[rd]^{p_2} & & E'\times E \ar[ld]_{p_2} \ar[rd]^{p_1} \\
E & & E' & & E
}
\]
with a cartesian square in the middle; here $p_{ij}$ denotes the projection to the product of $i$th and $j$th components, $q_1$ denotes the projection to the first component and $p_1,p_2$ are as above. We compute
\begin{align*}
T_{\psi^{-1}} \circ T_{\psi}(M) &= p_{1!}\left(p_2^\ast\, p_{2!}\,(p_1^\ast M \otimes \mu^\ast\,\CL_\psi) \otimes \mu^\ast\,\CL_{\psi\i}\right) \\
&\simeq p_{1!}\left(\pr_{12!}\,\pr_{23}^\ast\,(p_1^\ast\, M \otimes \mu^\ast\,\CL_\psi) \otimes \mu^\ast\,\CL_{\psi\i} \right) \\
&\simeq p_{1!}\, \pr_{12!}\, \left( pr_{23}^\ast\,(p_1^\ast M \otimes \mu^\ast\,\CL_\psi) \otimes \pr_{12}^\ast\,\mu^\ast \,\CL_{\psi\i} \right) \\
&\simeq q_{1!}\,\pr_{13!}\left(\pr_{12}^\ast\,\mu^\ast\CL_{\psi\i} \otimes \pr_{23}^\ast\mu^\ast\CL_\psi \otimes \pr_{13}^\ast\,q_2^\ast M \right) \\
&\simeq q_{1!}\left( pr_{13!}\,(\pr_{12}^\ast\,\mu^\ast\CL_{\psi\i} \otimes \pr_{23}^\ast\mu^\ast\CL_\psi) \otimes q_2^\ast M \right)
\end{align*}
where the second equation follows by the proper base change theorem for the cartesian square of the above diagram, the third and the fifth are by projection formula, the fourth follows from $p_1 \pr_{12} = q_1 \pr_{13}$ and $q_2\pr_{13} = p_1 \pr_{23}$. Now, let $X_1$ (resp. $X_2$) denote the pullback of the Lang--Steinberg covering of $W$ along $\mu\circ\pr_{12} \colon E \times E' \times E \to W$ (resp. along $\mu\circ\pr_{23}$). Let $X = X_1 \times_{E\times E' \times E} X_2$. This is a connected Galois cover of $E\times E' \times E$ with Galois group $W(\BF_q)^2$. We have the cartesian diagram
\[
\xymatrix{
E\times E' \times E \ar[r]^{\alpha} \ar[d]^{\pr_{13}} & E\times E' \ar[d]^{p_1}\\
E\times E \ar[r]^\beta & E
}
\]
where $\alpha(x,y,z) = (x-z,y)$ and $\beta(x,z) = x-z$. By \cite[Lemma I.5.10]{KW} we have $\pr_{12}^\ast\,\mu^\ast\CL_{\psi\i} \otimes \pr_{23}^\ast\mu^\ast\CL_\psi \simeq \alpha^\ast \mu^\ast\CL_\psi$. Thus, continuing the above computation, we have
\begin{align*}
T_{\psi^{-1}} \circ T_{\psi}(M) &\simeq q_{1!}\,\left( \pr_{13!}\,\alpha^\ast\mu^\ast \CL_\psi \otimes q_2^\ast M \right) \\
&\simeq q_{1!}\,\left( \beta^\ast \,p_{1!}\,\mu^\ast\CL_\psi \otimes q_2^\ast M \right),
\end{align*}
where the second equality holds by the proper base change theorem for the above diagram. Now, again by the proper base change theorem, the fiber of $p_{1!}\,\mu^\ast\,\CL_\psi$ at $x \in E$ is isomorphic to $p_{1!}^x \CL(\psi_x)$, where $p_1^x \colon \{x\} \times E' \to \{x\}$ is the restriction of $p_1$. By assumption $\CL(\psi_x)$ is non-trivial if $x \neq 0$, and hence by \cite[(1.1.3.4)]{Laumon_Fourier}, $p_{1!}^x \CL(\psi_x) = 0$ if $x\neq 0$. Thus, if $i \colon \{0\} \to E$ denotes the neutral section, $p_{1!}\,\mu^\ast\,\CL_\psi = i_\ast \ov\BQ_\ell[-2d](-d)$, so that continuing the above computation we get
\begin{align*}
T_{\psi^{-1}} \circ T_{\psi}(M) &\simeq q_{1!}\left(\beta^\ast \, i_! \ov\BQ_\ell[-2d](-d) \otimes q_2^\ast M \right) \\
&\simeq q_{1!} \left( \Delta_! \ov\BQ_\ell[-2d](-d) \otimes q_2^\ast M \right) \\
&\simeq q_{1!} \left( \Delta_! M \right)[-2d](-d) \\
&\simeq M[-2d](-d),
\end{align*}
where $\Delta$ denotes the diagonal of $E$, the first isomorphism is follows from the proper base change theorem, the third follows from the projection formula, and the last is due to the identity $q_1\circ\Delta = \id_E$.
\end{proof}

\subsection{} Now we specify to our case of interest. Let $\psi: \brk \to \ov\BQ_\ell^\times$ be an additive character, which is trivial over $\varpi \CO_k$ but nontrivial over $\CO_k$. Choose a sufficiently large integer $N \gg 0$, and view $W := \varpi^{-N} \CO_\brk / \varpi$ as an additive group scheme over $\BF_q$. Let $\CL_{\psi}$ be the rank one multiplicative local system on $W$ associated to $\psi$.

Let $\mathbf E$ be the set category of finitely generated $\CO_\brk$-modules such that $\varpi^N E = \{0\}$. Let $E \in \mathbf E$. Define \[E' = \hom_{\CO_\brk}(E, W).\] Then $E' \in \mathbf E$ and there is a natural pairing \[\mu: E \times E' \to W, \quad (x, f) \mapsto f(x).\] Notice that $E$ and $E'$ are commutative unipotent group schemes over $\BF_q$ of the same dimension. As in \Cref{subsec:W-Fourier}, we can define the Fourier transformation 
\begin{align*}
T_\psi^E \colon D_c^b (E,\ov\BQ_\ell) &\to D_c^b(E',\ov\BQ_\ell) \\
A &\mapsto p_{2!}(p_1^\ast A \otimes \mu^\ast\CL_\psi). \end{align*} 

By abuse of notation, we will write $T_\psi = T_\psi^E$ if $E$ is clear from the context.
\begin{proposition}\label{commutative}
    Let $E_1, E_2 \in \mathbf E$  and let $f: E_1 \to E_2$ be a morphism of $\CO_\brk$-modules. Let $f': E'_2\to E'_1$ be the dual morphism. For $\CF \in D_c^b(V_1)$ and $\CG\in D_c^b(V_2)$, there are natural isomorphisms
    \[T_{\psi}(f_!\CF)\cong f'^\ast T_{\psi}(\CF),\quad T_{\psi}(f^* \CF)\cong f'_! T_{\psi}(\CF)[2d](d). \]
    Here $d = \dim  E_2 - \dim E_1$.
\end{proposition}
\begin{proof}
    The first isomorphism is a standard result, see \cite[Proposition 6.9.13]{Achar}. The second isomorphism can be deduced from \Cref{prop:inversion}.
\end{proof}

\subsection{}\label{sec:Fourier-Witt-alg}
Let $\bfg$ be the Lie algebra of $G$ over $\brk$, whose dual is denoted by $\bfg^*$. Let $0 \le s \in \wt\BR$. We denote by $\bfg_{\bx, s}$ the associated Moy-Prasad $\CO_\brk$-submodule of $\bfg$, and set \[\bfg^*_{\bx, s} = \{X \in \mathbf  g^*; \<X, \bfg_{(-s)+}\> \in \varpi \CO_\brk\}\] where $\<-,-\>: \mathbf  g^* \times \mathbf  g \to \brk$ is the natural pairing. By abuse of notation we put \[\bfg_s := \bfg_{\bx, 0} / \bfg_{\bx, s+}, \quad \mathbf  g^*_{-s} := \mathbf  g^*_{\bx, -s} / \mathbf  g^*_{\bx, 0+}.\] Notice that $\bfg^*_{-s} \cong \hom_{\CO_\brk}(\bfg_s, W)$, where $W = \varpi^{-N} \CO_\brk / \varpi$ with $N \gg 0$. Then we have the Fourier transform functor \[T_{\psi}^{\bfg_{-s}^*} : D_c^b(\mathbf  g^*_{-s}) \to D_c^b(\bfg_s), \quad M \mapsto \pr_{1 !}(\pr_2^* M \otimes \<-, -\>^* \CL_{\psi}).\] 

%Applying \Cref{prop:inversion} to $W = \varpi^{-N} \CO_\brk / \varpi$, $E = \bfg_s$, $E' = \bfg^*_{-s}$ and $\mu = \langle,\rangle$, we get the following.

%\begin{corollary}\label{Fourierequiv}For any $M \in D_{G_s}(\bfg^*_{-s})$, the equation\[T_{-\psi}^{\bfg_s} \circ T_{\psi}^{\bfg^*_{-s}}(M) = M[-2\dim{\bfg_s}](-\dim{\mathbf{g}_r)}\]\[\BD \circ T_{\psi}^{\bfg^*_{-s}}(M) \cong T_{\psi}^{\bfg^*_{-s}} \circ \BD(M) [2\dim{\mathbf{g}_s}](\dim{\mathbf{g}_s)}\]canonically holds.\end{corollary}

\

Let notation be as in \Cref{subsec:parabolic}. Let $\bfp, \bfl, \bfn, \ov \bfn$ be the Lie algebras of $P, L, N, \ov N$ respectively. Put \[\bfp_{s, r} = \bfl_r + \bfn_{s:r} + \ov\bfn_{s+:r}, \quad \bfp^*_{-s, -r} = \bfl^*_{-r} + \bfn^*_{(-s)+} + \ov\bfn^*_{-s}.\] Note that $\bfp^*_{-s, -r} \cong (\bfg_r / \bfn_{s, r})'$ with $\bfn_{s, r} = \bfn_{s:r} + \ov\bfn_{s+:r}$. 
Consider the following Cartesian diagram \[ \tag{$\ast$}    
\xymatrix@R=1.5em@C=2.3em{
    & \bfp_{s,r} \ar[dl]_{p} \ar[dr]^{i} & \\
 \bfl_r \ar[dr]^-{j}& & \bfg_r \ar[dl]_{q} \\
    & \bfg_r/\bfn_{s,r}, & 
}\] where all the morphisms are defined in the natural way. Then can define the induction functor for Lie algebras:  
\[ \pInd^{\bfg_r}_{\bfp_{s,r}} := \Av^{G_r}_{P_{s,r}!} \circ i_! \circ p^\ast \circ \mathrm{Infl}^{P_{s,r}}_{L_r}: D_{L_r}(\bfl_{r})\to D_{G_r}(\bfg_{r}). \] Similarly, we can define \[\pInd_{\bfp^*_{-s, -r}}^{\bfg^*_{-r}}:= \Av^{G_r}_{P_{s,r}!} \circ {q'}_! \circ {j'}^* \circ \Infl^{P_{s,r}}_{L_r}: D_{L_r} (\bfl^*_{-r})  \to D_{G_r} (\bfg^*_{-r})\] by using the following diagram dual to $(*)$ \[ \tag{$\ast\ast$}
\xymatrix@R=1.5em@C=2.3em{
    & \bfp^*_{-s,-r} \ar[dl]_{j'} \ar[dr]^{q'} & \\
    \bfl^*_{-r} \ar[dr]^-{p'}& & \bfg^*_{-r} \ar[dl]_{i'} \\
    & (\bfp_{s, r})', & 
}\] As $(*)$ is Cartesian, so is $(**)$.
\begin{lemma}\label{Lieindu}
\begin{itemize} We have $T_{\psi} \circ \pInd^{\bfg^*_{-r}}_{\bfp^*_{-s,-r}} = \pInd^{\bfg_r}_{\bfp_{s,r}} \circ T_{\psi}[-2\dim N_{s, r}](-\dim N_{s,r})$.
%The restriction of $\pInd^{\bfg_r}_{\bfp_{s,r}}$ on $D_{L_r}((\bfl_{r})_\unip)\to D_{G_r}((\bfg_{r})_\unip) $ coincides with that of $\pInd^{G_r}_{P_{s,r}}$.
\end{itemize}
\end{lemma}
\begin{proof}
By \Cref{commutative} and the proper base change theorem for $(**)$, we have \begin{align*}
    &\quad\ \pInd^{\bfg_r}_{\bfp_{s,r}} \circ T_\psi \\ 
    &=\Av^{G_r}_{P_{s,r}!} \circ i_! \circ p^* \circ \Infl_{L_r}^{P_{s,r}} \circ T_\psi \\
    &=\Av^{G_r}_{P_{s,r}!} \circ i_! \circ p^* \circ T_{\psi} \circ \Infl_{L_r}^{P_{s,r}} \\ 
    &=\Av^{G_r}_{P_{s,r}!} \circ i_! \circ T_{\psi}  \circ p'_! \circ \Infl_{L_r}^{P_{s,r}} \\
    &=\Av^{G_r}_{P_{s,r}!} \circ T_{\psi} \circ {i'}^*  \circ p'_! \circ \Infl_{L_r}^{P_{s,r}} [2\dim N_{s,r}] (\dim N_{s,r}) \\
    &=T_{\psi} \circ \Av^{G_r}_{P_{s,r}!} \circ {i'}^*  \circ p'_! \circ \Infl_{L_r}^{P_{s,r}} [2\dim N_{s,r}] (\dim N_{s,r}) \\
    &=T_{\psi} \circ \Av^{G_r}_{P_{s,r}!} \circ q'_!\circ {j'}^* \circ \Infl_{L_r}^{P_{s,r}} [2\dim N_{s,r}](\dim N_{s,r}) \\
    &=T_\psi \circ \pInd^{\bfg^*_{-r}}_{\bfp^*_{-s,-r}} [2\dim N_{s,r}] (\dim N_{s,r}),
\end{align*} where the fifth identity follows from \Cref{prop:inversion} and that $T_\psi$ commutes with the right adjoint $\For_{P_{s,r}}^{G_r}$ to $\Av^{G_r}_{P_{s,r}!} $. The proof is finished.
\end{proof}

Recall that $\fkl = L_{r:r} \cong \bfl_{\bx, r} / \bfl_{\bx, r+}$.
\begin{lemma} \label{ind-delta}
    Let $X \in \bfl^*_{-r}$ such that $X |_{\fkl} \in \fkl^*$ is $(L, G)$-generic. Then $\pInd_{\bfp^*_{-s, -r}}^{\bfg^*_{-r}}(\d_{L_r \cdot X}) = \d_{G_r \cdot X} [2 \dim N_{s, r}]$.
\end{lemma}
\begin{proof}
    Let \begin{align*} \wt{\bfg}^*_{-r} &= \{(x, h P_{s, r}) \in \bfg^*_{-r} \times G_r/P_{s,r}; h\i \cdot x \in \bfp^*_{-s, -r}\} \\ \wh{\bfg}^*_{-r} &= \{(x, h) \in \bfg^*_{-r} \times G_r; h\i \cdot x \in \bfp^*_{-s, -r}\}. \end{align*} Then we have natural maps \begin{align*}
    &\eta: \wh \bfg^*_{-r} \to \bfl^*_{-r}, \quad & (x, h ) &\mapsto p(h\i \cdot x); \\ 
    &\alpha: \wh \bfg^*_{-r} \to \wt\bfg^*_{-r}, \quad & (x, h) &\mapsto (x,hP_{s,r});\\
    &\pi: \wt\bfg^*_{-r} \to \bfg^*_{-r}, \quad & (x, h P_{s, r}) &\mapsto x.
\end{align*} As in \Cref{otherconstruct}, by forgetting the equivariant structure we have \[\pInd_{\bfp^*_{-s, -r}}^{\bfg^*_{-r}}(\d_{L_r \cdot X}) \cong \pi_! \wt{\eta^* \d_{L_r \cdot X}} [2\dim N_{s, r}].\] Thus it suffices to show that $\eta\i(Z) \cong P_{s, r}$ for each $Z \in L_r \cdot X$. This follows from the assumption that $Z|_\fkl$ is $(L, G)$-generic.
\end{proof}

\subsection{}

Let $(\bfg_r)_\nilp$ and $(G_r)_\unip$ be the sets of nilpotent elements and unipotent elements in $\bfg_r$ and $G_r$ respectively. Assume that $p$ is sufficiently large. By \cite{DR09} and \cite{BC24}, there is an $F$-equivariant and $G_r$-equivariant bijection \[\log: (G_r)_\unip \overset {1:1} \longrightarrow (\bfg_r)_\nilp.\] The inverse of $\log$ is denoted by $\exp: (\bfg_r)_\nilp \to (G_r)_\unip$.

Let $T$, $\phi$, $r$, $(G^i, \phi_i, r_i)_{-1 \le i \le d}$, $B = U T$ and $\CI_r = \CI_{\phi, U, r}$ be as in \Cref{sec:I-induction}. Let $\mathbf  t$, $\mathbf  g^i$ and $\mathbf  z^i$ be the Lie algebras of $T$, $G^i$ and $Z(G^i)$ respectively. Their dual spaces are denoted by $\mathbf  t^*$, $(\mathbf  g^i)^*$ and $(\mathbf  z^i)^*$ respectively.

Assume $q$ is sufficiently large. Then there exists a regular element $X \in (\bft_{-r}^*)^F$ (that is, the centralizer of $X$ in $G_r$ is $T_r$) such that \[\psi(\<X, Y\>) = \phi(\exp(Y)), \quad \forall\ Y \in \bft_{0+:r}^F.\]  Choose $X_i \in ((\mathbf  z^i)^*_{-r_i})^F$ such that \[\psi(\<X_i, Y\>) = \phi_i(\exp(Y)), \quad \forall\ Y \in \mathbf t_{0+:r}^F,\ 0 \le i \le d.\] Then $X = X_{-1} + \sum_{i=0}^d X_i$ for some $X_{-1} \in (\bft^*_0)^F$ which is regular for $G^0$.

Let $\delta_{X_i}$ be the constant sheaf over $\{X_i\}$. Let $\fkL_i = T_\psi^{(\bfg^i)^*_{-r_i}}(\delta_{X_i})$, which is a rank one local system on $\bfg^i_{r_i}$. Put $\fkL = \bigotimes_{t=-1}^d q_{-1,t}^*\fkL_t$, where $q_{-1,t}: \bft_r \to \bfg^i_{r_i}$ is the natural projection. 

For $\CF \in D(G_r^i)$ we will write $\exp^*\CF = \exp^*(\CF|_{(\bfg^i_r)_\unip}) \in D((\bfg^i_r)_\nilp)$ for simplicity. Then we have \[(\exp^*\CL_\phi) |_{(\bft_r)_\nilp} \cong \fkL|_{(\bft_r)_\nilp}.\] Here $\CL_\phi$ is the rank one multiplicative local system on $T_r$ associated to the character $\phi$.

\begin{theorem}\label{Sheafsprhy}
    There is a natural isomorphism \[\exp^*\pInd_{\CI_r}^{G_r}(\CL_\phi) \cong T_{\psi}^{\mathbf  g^*_{-r}}(\delta_{G_r \cdot X})|_{(\bfg_r)_\nilp}[2 M_\phi](\frac{1}{2}M_\phi),\] where $M_\phi = \sum_{i=0}^d (\dim G^i_{r_{i-1}} - \dim G^{i-1}_{r_{i-1}})$.
\end{theorem}
\begin{proof}
    Let $\pInd_{\bfi_r}^{\bfg_r}$ be the Lie algebra version of $\pInd_{\CI_r}^{G_r}$ as in \Cref{sec:I-induction}, where $\bfi_r$ denotes the Lie algebra of $\CI_r$. As $(\exp^*\CL_\phi) |_{(\bft_r)_\nilp} \cong \fkL|_{(\bft_r)_\nilp}$, by the proper base change theorem and Lie algebra version of \Cref{multi-step} we have \[\exp^*\pInd_{\CI_r}^{G_r}(\CL_\phi) \cong \pInd_{\bfi_r}^{\bfg_r}(\fkL) |_{(\bfg_r)_\nilp} \cong \Psi_d \cdots \Psi_1 \Psi_0(\fkL_{-1}) |_{(\bfg_r)_\nilp},\] where $\Psi_i: D(\bfg^{i-1}_{r_{i-1}}) \to D(\bfg^i_{r_i})$ is given by $\fkF \mapsto \fkL_i \otimes \varepsilon_i^* \pInd_{\bfp^i_{s_{i-1}, r_{i-1}}}^{\bfg^i_{r_{i-1}}} \fkF$ with $\varepsilon_i: \bfg^i_{r_i} \to \bfg^i_{r_{i-1}}$ the natural projection. Set $\fkF_i = \Psi_i \cdots \Psi_1 \Psi_0(\fkL_{-1}) $. It suffices to show \[\fkF_i \cong T_\psi^{(\bfg^i)^*_{-r_i}}(\delta_{G^i_{r_i} \cdot X_{\le i}})[2n_i](n_i/2).\] Here $X_{\le i} = \sum_{l=-1}^i X_l$ and $n_i = \sum_{l=0}^i (\dim G^i_{r_{i-1}} - \dim G^{i-1}_{r_{i-1}}) = 2\sum_{l=0}^i \dim N^l_{s_{l-1}, r_{l-1}}$.
    
    We argue by induction on $i$. If $i = -1$, the statement follows by definition that $\fkL_{-1} = T_\psi^{\bft^*_{-r}}(\delta_{X_{-1}})$. We assume the statement holds for $i \ge -1$. Then
    \begin{align*}
        \fkF_{i+1} &\cong \fkL_{i+1}  \otimes \varepsilon_{i+1}^* \pInd_{\bfp^{i+1}_{s_i, r_i}}^{\bfg^{i+1}_{r_i}} \fkF_i  \\ &\cong \fkL_{i+1}  \otimes \varepsilon_{i+1}^* \pInd_{\bfp^{i+1}_{s_i, r_i}}^{\bfg^{i+1}_{r_i}} T_\psi^{(\bfg^i)^*_{-r_i}} \d_{G^i_{r_i} \cdot X_{\le i}}) [2n_i](n_i/2) \\ &\cong \fkL_{i+1} \otimes \varepsilon_{i+1}^* T_\psi^{(\bfg^{i+1})^*_{-r_i}} \pInd_{(\bfp^{i+1})^*_{-s_i, -r_i}}^{(\bfg^{i+1})^*_{-r_i}} (\d_{G^i_{r_i} \cdot X_{\le i}}) [n_{i+1} + n_i](n_{i+1}/2) \\  &\cong \fkL_{i+1} \otimes \varepsilon_{i+1}^* T_\psi^{(\bfg^{i+1})^*_{-r_i}} (\d_{G^{i+1}_{r_i} \cdot X_{\le i}}) [2n_{i+1}](n_{i+1}/2) \\ &\cong \fkL_{i+1} \otimes  T_\psi^{(\bfg^{i+1})^*_{-r_{i+1}}} (\d_{G^{i+1}_{r_i} \cdot X_{\le i}}) [2n_{i+1}](n_{i+1}/2) \\ &\cong T_\psi^{(\mathbf  g^{i+1})^*_{-r_{i+1}}}(\d_{X_{i+1}}) \otimes  T_\psi^{(\bfg^{i+1})^*_{-r_{i+1}}} (\d_{G^{i+1}_{r_i} \cdot X_{\le i}}) [2n_{i+1}](n_{i+1}/2) \\  &\cong T_\psi^{(\bfg^{i+1})^*_{-r_{i+1}}} (\d_{G^{i+1}_r \cdot X_{\le i+1}}) [2n_{i+1}](n_{i+1}/2),    
    \end{align*}
where the second is by  induction hypothesis, the third follows from \Cref{Lieindu}, the fourth follows from \Cref{ind-delta} and the fifth follows from \Cref{commutative}. The induction procedure is finished.
\end{proof}

Let $C((\bfg^*_{-r})^F)$ and $C(\bfg_{-r}^F$ be the spaces of functions on $(\bfg^*_{-r})^F$ and $\bfg_r^F$ respectively. Let ${\rm T}_\psi^{\bfg_{-r}^*}: C((\bfg^*_{-r})^F) \to C(\bfg_r^F)$ denotes the classical Fourier transformation of functions given by \[{\rm T}_\psi^{\bfg_{-r}^*}(f)(Y) = \sum_{Z \in (\bfg^*_{-r})^F} f(X) \psi(\<Z, Y\>).\] Note that for any Weil sheaf $\CF$ in $D_c^b(\bfg_{-r}^*)$ we have $\chi_{T_\psi^{\bfg_{-r}^*}\CF} = {\rm T}_\psi^{\bfg_{-r}^*}(\chi_\CF)$.
\begin{corollary}\label{cor:Springer}
    Assume $p$ and $q$ are sufficiently large and $T$ is elliptic, for any $u\in (\bfg_r)^F_{\nilp}$ we have
    \[q^{\frac{1}{2}M_\phi} \cdot R^{G_r}_{T_r}(\phi)(\exp(u)) = T_{\psi}^{\bfg^*_{-r}}(1_{G_r^F \cdot X})(u).\] Here $1_{G_r^F \cdot X}$ is the characteristic function for $G_r^F \cdot X$.
\end{corollary}
\begin{proof}
As $q \gg 0$, we can replace $\phi_{-1}$ with another depth zreo character $\phi'_{-1}$ which is regular for $G^0$. This yields a regular character 
\[
\phi_{\reg} := \phi'_{-1} \cdot \prod_{0 \le i \le d} \phi_i|_{T^F_r}
\]
which gives a local system $\CL_{\rm reg}$ on $T_r$ such that $\mathcal{L}_{\mathrm{reg}}|_{(T_r)_{\unip}} \cong \mathcal{L}_{\phi}|_{(T_r)_{\mathrm{unip}}}$. By the proper base change theorem and \Cref{Sheafsprhy} we have
\[\exp^*\pInd_{\CI_r}^{G_r}(\CL_{{\reg}}) \cong \exp^*\pInd_{\CI_r}^{G_r}(\CL_\phi) \cong T_{\psi}^{\bfg^*_{-r}}(\d_{G_r \cdot X})|_{(\bfg_r)_{\nilp}}[2 M_\phi](\frac{1}{2} M_\phi).\] By taking Frobenius trace we have \begin{align*} 
R^{G_r}_{T_r}(\phi)(\exp(u)) &= R^{G_r}_{T_r}(\phi_{{\reg}})(\exp(u)) \\
&= \chi_{\pInd^{G_r}_{\CI_r}(\CL_{\reg})}(\exp(u)) = q^{-\frac{1}{2} M_\phi} \cdot {\rm T}_{\psi}^{\bfg^*_{-r}}(1_{G_r^F \cdot X})(u),\end{align*} where the first equality is due to \cite[Theorem 1.6]{Nie_24} and \cite[Theorem 4.2]{DeligneL_76} (see also \cite[Theorem 6.5]{CO25}), and the second one is due to \Cref{mainthm} and that $N_\phi = 2\dim \CI_r - \dim G_r$. The proof is finished. 
\end{proof}

\subsection{Relation to Kirillov's orbit method}

As a further application we observe that Springer's hypothesis (\Cref{cor:Springer}) implies a relation between (deep level) Deligne--Lusztig induction and Kirillov's orbit method, as conjectured in \cite[Conjecture 8.4]{IvanovNie_24}. In the relevant setting, the orbit method parametrizes the set $\widehat\Gamma$ of irreducible $\Gamma$-representations, for a sufficiently nice pro-$p$-group $\Gamma$, in terms of coadjoint orbits in the dual Lie algebra:

\begin{theorem}[Theorem 2.6 in \cite{BoyarchenkoS_08}]\label{thm:orbit_method}
Assume $p > 2$ and $\Gamma$ is either a uniform pro-$p$-group or a $p$-group of nilpotence class $<p$. Then there exists a bijection $\Omega \leftrightarrow \rho_\Omega$  between $\Gamma$-orbits $\Omega \subseteq ({\rm Lie}\ \Gamma)^*$ and $\widehat \Gamma$, characterized by
\[ 
\tr(g,\rho_\Omega) = \frac{1}{\#\Omega^{1/2}} \cdot \sum_{f\in \Omega} f(\log(g)).
\]
\end{theorem}

Fix $0+\leq r\leq \infty$. Assume that $T$ is elliptic. Let $\mathbf{g}_{0+:r}$ (resp. $\mathbf{t}_{0+:r}$) denote the Lie algebra of $G_{0+:r}$ (resp. $T_{0+:\infty}$). For a locally profinite group $\Gamma$, write $\Gamma^\ast = \Hom_{\rm cont}(\Gamma,\ov\BQ_\ell^\times)$. Recall from \cite[\S8]{IvanovNie_24}, that (a minor variation of) Deligne--Lusztig induction gives a map
\[
R_{\rm log} \colon (\mathbf{t}_{0+:r}^F)^\ast \stackrel{\log^\ast_{\mathbf{t}}}{\to} (T_{0+:r}^F)^\ast \to \widehat{G_{0+:r}^F},
\]
where the second map is given by $\chi \mapsto (-1)^{s_\chi} H_c^{s_\chi}(Y,\ov\BQ_\ell)[\chi]$, where $Y_r$ is the preimage of $(\ov U \cap FU)_{0+:r}$ under the Lang map $g \mapsto g^{-1}F(g) \colon G_{0+:r} \to G_{0+:r}$. (In \cite{IvanovNie_24}, this map is only defined for $T$ Coxeter, but by \cite[Theorem 1.6]{IvanovNie_25} this extends to all elliptic tori $T$.) On the other hand, whenever the assumptions of Theorem \ref{thm:orbit_method}  apply to $G_{0+:r}^F$, the orbit method induces the map
\[
\rho \circ \delta^\ast \colon (\mathbf{t}_{0+:r}^F)^\ast \stackrel{\delta^\ast}{\hookrightarrow} (\mathbf{g}_{0+:r}^F)^\ast \stackrel{\rho}{\twoheadrightarrow} \widehat{G_{0+:r}^F},
\]
where the first map is the dual of the natural projection $\delta \colon \mathbf{g}_{0+:r}^F \twoheadrightarrow \mathbf{t}_{0+:r}^F$, the second map is the natural projection, and the third map is given by \Cref{thm:orbit_method}. Now we can prove \cite[Conjecture 8.4]{IvanovNie_24}:

\begin{corollary}\label{cor:orbit}
Suppose $p>2$, $q$ is sufficiently large and $G_{0+:r}^F$ is either uniform or has nilpotence class $<p$. Then $R_{\rm log} = \rho \circ \delta^\ast$.
\end{corollary}
\begin{proof}
Let $X \in (\mathbf{t}_{0+:r}^F)^\ast$ with image $\phi = \log_{\mathbf{t}}^\ast(a) \colon T_{0+:r}^F \to \ov\BQ_\ell^\times$. 
%Let $X_r^{(T)}$ be the subvariety of $X_r$ given by ... . By \cite[Lemma 4.3]{IvanovNie_24} we have $X_r^{(T)} = \coprod_{\gamma \in G_{0:0+}^F/T_{0:0+}^F} \gamma (X_r \cap T_r G_{0+:r})$. 
As $q\gg 0$, there exists a lift $\tilde\phi \colon T_r^F \to \ov\BQ_\ell^\times$ of $\phi$, such that $\phi_{-1}$ is regular. As $G_{0+:r}^F$-representations we have for any $i\geq 0$,
\[
H_c^i(X_r)[\tilde\phi]|_{G_{0+:r}^F} \simeq H_c^i(X_r \cap T_r G_{0+:r})[\tilde\phi]|_{G_{0+:r}^F} \simeq H_c^i(Y_r)[\phi],
\]
where the first isomorphism follows from \cite[Corollary 1.5]{IvanovNie_25} and the second from \cite[Lemma 4.3]{IvanovNie_24}. As by \cite[Theorem 1.6]{IvanovNie_25}, the cohomology of $Y_r$ is concentrated in the single degree $s_{\phi,r}$, we deduce that for any $Y \in \mathbf{g}_{0+:r}^F$ with $g = \exp(Y) \in G_{0+:r}^F$,
\begin{align*} 
{\rm tr}(g, R_{\rm log}(X)) &= {\rm tr}(g, H^\ast_c(Y_r,\ov\BQ_\ell)[\phi]) \\ 
&= {\rm tr}(g, H^\ast_c(X_r,\ov\BQ_\ell)[\tilde\phi]) \\
&= R_{T_r}^{G_r}(\tilde\phi)(g) \\
&= q^{-\frac{1}{2}M_\phi} {\rm T}_\psi^{\bfg_{-r}^\ast}(1_{G_r^F \cdot X})(Y) \\
&= q^{-\frac{1}{2}M_\phi} \sum_{f \in G_r^F \cdot X} f(Y) \\
&= q^{-\frac{1}{2}M_\phi} \#(G_r^F \cdot X)^{\frac{1}{2}} \cdot {\rm tr}(g, \rho_{G_r^F \cdot X}),
\end{align*}
where the fourth equation is by \Cref{cor:Springer}, the fifth equation follows from the fact that taking Frobenius trace commute with the Fourier transform (cf. \cite[\S9.1.3]{CO25}), and the last equation follows from \Cref{thm:orbit_method}. Now the result follows from  $R_{\rm log}(X)$ and $\rho_{G_r^F \cdot X}$ both being irreducible $G_r^F$-modules.
\end{proof}

\section{Proof of \Cref{one-step}} \label{sec:one-step-proof}
In this section, we prove \Cref{one-step} by following the ideas of \cite{BC24}. Let notation be as in \Cref{sec:pre} and \Cref{sec:one-step}.

\begin{lemma}[{\cite[Lemma 4.7]{BC24}}] \label{restriction}
    The following statements holds:

    (1) $L_\psi$ is a multiplicative local system on $\fkl$;

    (2) Each $\CF \in D_{L_r}^\psi(L_r)$ is $(\fkl, \CL_\psi)$-equivariant. In particular, $i_\fkl^*\CF$ is complex of direct sums of $\CL_\psi$.
\end{lemma}

\begin{proposition}\label{ind}
    The induction restrict to funtors: \[\pInd_{P_{s, r} !}^{G_r},\ \pInd_{P_{s, r} *}^{G_r}: D_{L_r}^\psi(L_r) \to D_{G_r}^\psi(G_r).\]
\end{proposition}
\begin{proof}
    Let $\CF \in D_{L_r}^\psi(L_r)$. By duality it suffices to show $M := \pInd_{P_{s, r} !}^{G_r} \CF \in D_{G_r}^\psi(G_r)$. 

    We view the natural quotient map $q: G_r \to G_r / \fkg$ as a vector bundle over $G_r / \fkg$. Let $U \subseteq  G_r / \fkg$ be any open subset such that $q\i(U) \cong U \times \fkg$. It suffices to show that $\FT_{\CL_\psi}(M |_{q\i(U)})$ has support in $U \times -G_r \cdot X_\psi$. 

    Consider the following diagram \[\xymatrix{ &
    {\wh Z} \ar[r]^{\a'} \ar[d]_{\wh \pr_1} & {\wt Z} \ar[r]^{\pi'} \ar[d]_{\wt\pr_1} & {G_r \times_U G_r'} \ar[d]_{\pr_1} \ar[r]^{\pr_2} & G_r'\\
    L_r& {\wh G_r} \ar[l]_\eta \ar[r]^\a & {\wt G_r} \ar[r]^\pi & G_r,
    }\] where $G_r'$ denotes the dual vector bundle of $q: G_r \to G_r/\fkg$ and the two squares are Cartesian. We view $U \subseteq G_r$ as the zero section of $q$ and let \[\wt U = \pi\i(U), \quad \wh U = \a\i(\wt U).\] As $r > 0$, we have $\fkg \subseteq P_{s, r}$ and hence \[\wt Z \cong \wt U \times \fkg \times \fkg^*, \quad \wh Z \cong \wh U \times \fkg \times \fkg^*.\]

    By the proper base change theorem and the projection formula we have \begin{align*} \tag{i}
        \FT_\CL(M)|_{q\i(U)} &\cong (\pr_2)_! \pr_1^*(\pInd_{P_{s, r} !}^{G_r} \CF) \otimes  \k^*\CL) \\
        & \cong (\pr_2)_!(\pr_1^* \pi_! \wt{\eta^*M }\otimes \k^*\CL) \\
        &\cong  (\pr_2)_! ((\pi')_! \wt\pr_1^*\wt{\eta^*M } \otimes  \k^*\CL) \\
        &\cong (\pr_2\circ \pi')_!( \wt\pr_1^*\wt{\eta^*M} \otimes (\k \circ \pi')^*\CL).
    \end{align*}

    Consider the Cartesian diagram \[ \tag{ii} \xymatrix{
    \wh Z \ar[r]^\a \ar[d]_{\wh\pr_{13}} & \wt Z \ar[d]^{\wt\pr_{13}} \\
    \wh U \times \fkg^* \ar[r]^{\a'\times \id} & \wt U \times \fkg^*,
    }\] where $\wh\pr_{13}$ and $\wh\pr_{13}$denote the natural projections.

    Note that $\pr_2 \circ \pi' = \g \circ \wt\pr_{1,3}$, where $\g: \wt U \times \fkg^* \to G_r'$ be given by $(z, hP_{s, r}, X) \mapsto (z, X)$. In view of (i) and (ii), to show $\FT_\CL(M|_{q\i(U)})$ vanishes outside $U \times -G_r \cdot X_\psi$, it suffices to show \[(\g \times \id)^* (\wt\pr_{13})_! ( \wt\pr_1^*\wt{\eta^*M} \otimes (\k \circ \pi')^*\CL)\] vanishes outside $\wh U \times -G_r \cdot X_\psi$. By the the proper base change theorem for (ii) we have \begin{align*}
         &\quad\ (\g \times \id)^* (\wt\pr_{13})_! ( \wt\pr_1^*\wt{\eta^*M} \otimes (\k \circ \pi')^*\CL) \\ &\cong (\pr_{13})_! \left( (\wt\pr_1 \circ \a')^*\wt{\eta^*\CF} \otimes (\k \circ \pi' \circ \a')^* \CL \right) \\ &\cong (\pr_{13})_! \left( (\a \circ \wh\pr_1)^*\wt{\eta^*\CF} \otimes (\k \circ \pi' \circ \a')^* \CL \right) \\  &\cong (\pr_{13})_! \left(  \wh\pr_1^* \a^*\wt{\eta^*\CF} \otimes (\k \circ \pi' \circ \a')^* \CL \right) \\  &\cong (\pr_{13})_! \left( \wh\pr_1^* \eta^*\CF \otimes (\k \circ \pi' \circ \a')^* \CL \right) \\  &\cong (\pr_{12})_! \left( (\eta \circ \wh\pr_1)^*\CF \otimes (\k \circ \pi' \circ \a')^* \CL \right).
    \end{align*}

    Now we fix a point $\xi = (z, h, X) \in \wh U \times \fkg^*$ with $X \notin -G_r(X_\psi)$. Identify the fiber $\pr_{12}\i(\xi)$ with $\fkg$.  Put \begin{align*}a := (\eta \circ \wh \pr_1)|_{\pr_{12}\i(\xi)}: &\fkg \to L_r, \quad Y \mapsto p(h\i \cdot Y); \\ b := (\k \circ \pi' \circ \a')|_{\pr_{12}\i(\xi)}: &\fkg \to \BG_a, \quad Y \mapsto \k(Y, X_\psi).\end{align*} It remains to show \[R\G_c(a^*\CF \otimes b^*\CL) = 0.\] By \Cref{reduction} we may assume $\CF|_\fkl = \CL_\psi$. Thus \[a^*\CF \otimes b^*\CL \cong \th^*\CL_\psi,\] where $\th : \fkg \to \BG_a$ is defined by $Y \mapsto \k(Y, h \cdot X_\psi + X)$. As $X \notin - G_r \cdot X_\psi$, the map $\th$ is nontrivial. It follows that $R\G_c(a^*\CF \otimes b^*\CL) \cong R\G_c(\th* \CL_\psi) = 0$ as desired. The proof is finished.
\end{proof}

\subsection{} \label{subsec:ind-Lie}
By abuse of notation let $i: \fkp \hookrightarrow \fkg$ and $p: \fkp \to \fkl$ be the natural inclusion and projection respectively. Define \[\mathfrak{pInd}_{P_{s, r}!}^{G_r}:= \Av_{P_{s, r}!}^{G_r} \circ i_! \circ p^* \circ \Infl_{L_r}^{P_{s, r}}: D_{L_r}(\fkl) \to D_{G_r}(\fkg).\] Similarly, we introduce the varieties
\begin{align*}
    \wt \fkg &:= \{(g, hP_{s,r}) \in \fkg \times G_r/P_{s,r} : h^{-1} g h \in \fkp\}, \\
    \widehat \fkg &:= \{(g, h) \in \fkg \times G_r : h^{-1} g h \in \fkp\},
  \end{align*} together with the morphisms \begin{align*}
    &\eta: \wh \fkg \to \fkl, \quad & (g, h ) &\mapsto p(h\i g h); \\ 
    &\alpha: \wh \fkg \to \wt \fkg, \quad & (g,h) &\mapsto (g,hP_{s, r});\\
    &\pi: \wt \fkg \to \fkg, \quad & (g, h P_{s, r}) &\mapsto g.
\end{align*}
Similar to \Cref{otherconstruct}, for $\CF \in D_{L_r}(\fkl)$ we have \[\mathfrak{pInd}_{P_{s, r} !}^{G_r}(\CF) \cong \pi_! \wt{(\eta^*\CF}[2 \dim(G_r/P_{s, r})] \text{ for } \CF \in D_{L_r}(\fkl).\]

\begin{proposition} \label{ind-Lie}
    We have $\mathfrak{pInd}_{P_{s, r}!}^{G_r} (\CL_\psi[\dim \fkl]) \cong \CF_\psi[\dim \fkg]$.
\end{proposition}
\begin{proof}
    By \cite[Remark 3.9]{BC24} we have $\mathfrak{pInd}_{P_{s, r}!}^{G_r} \cong \mathfrak{pInd}_{P_{0, 0}!}^{G_0}$. Then statement then follows by \cite[Lemma 4.5]{BC24}.
\end{proof}

\subsection{}
We adopt the notation in \Cref{subsec:aff-root}. For $f, f' \in \tPhi^+$ we write $f \sim_L f'$ if $f'(\bx) = f(\bx)$ and $\a_{f'}-\a_{f} \in \BZ  \Phi_L$. This defines an equivalence relation on $\tPhi^+$, and we denote by $[f]$ the equivalence class of $f$. Note that the linear order $\le$ (attached to the fixed Borel subgroup $T U = B \subseteq P$) induces a linear order on $\tPhi^+ / \sim_L$, which we still denote by $\le$.

Let $\Psi \subseteq \tPhi^+$ be the set of affine roots appearing in $p\i(\fkl) = N_{s, r} \fkl$. For $f \in \Psi$ with $f \le r$ we set $p\i(\fkl)^{\ge [f]} = \prod_{f' \in \Psi, [f] \le [f']} G_r^{f'}$ and define \begin{align*}  \wh G_r^{\ge [f]} &= \{(g, h) \in G_r \times G_r; h\i g h \in p\i(\fkl)^{\ge [f]}\}; \\ \wt G_r^{\ge [f]} &= \{(g, h P_{s, r}) \in G_r \times (G_r / P_{s, r}); h\i g h \in p\i(\fkl)^{\ge [f]}\}; \\ \wt Y_r^{\ge [f]} &= \{(g, h L_r G_r^{\ge [r-f]}) \in G_r \times (G_r / L_r G_r^{\ge [r-f]}); h\i g h \in p\i(\fkl)^{\ge [f]}\}.\end{align*} 
Let $\wh G_r^{\ge [f], *} = \wh G_r^{\ge [f]} \sm \wh G_r^{> [f]}$, $\wt G_r^{\ge [f], *} = \wt G_r^{\ge f} \sm \wt G_r^{> [f]}$ and $\wt Y_r^{\ge f, *} = \wt Y_r^{\ge f} \sm \wt Y_r^{> [f]}$. Note that $\a\i(\wt G_r^{\ge [f]}) = \wh G_r^{\ge [f]}$. Let $\b_f: \wt G_r^{\ge [f], *} \to \wt Y_r^{f, *}$ be the natural projection. Put $\eta_f = \eta|_{\wh G_r^{\ge [f],*}}$ and $\pi_f = \pi|_{\wt G_r^{\ge [f],*}}$

\begin{lemma} \label{vanish}
    Let $f \in \Psi$ such that $f < r$. Then  $(\pi_f)_!\wt{\eta_f^* \CF} = 0$ for $\CF \in D_{L_r}^\psi(L_r)$.
\end{lemma}
\begin{proof}
    By \Cref{restriction} we may assume $\CF = \CL_{\psi, r}$, and it suffices to show $(\b_f)_!\wt{\eta_f^* \CL_{\psi, r}} = 0$. Let $\xi = (g, h L_r G_r^{\ge [r-f]}) \in \wt Y_r^{\ge f, *}$. Then \[\b_f\i(\xi) = \{(g, h L_r G_r^{\ge [r-f]}/P_{s, r})\} \cong  G_r^{D_f},\] where $D_f$ is the set of affine roots appearing in $ L_r G_r^{\ge [r-f]}/P_{s, r}$. Denote by $j: \b_f\i(\xi) \cong  \{(g, h G_r^{D_f})\} \hookrightarrow \wh G$ be the natural embedding. Consider the Cartesian diagram \[\xymatrix{
    \b_f\i(\xi) \ar[r]^{\wt i_\xi} \ar[d] & \wt G_r^{\ge [f], *} \ar[d]_{\b_f} \\ \{\xi\} \ar[r]^{i_\xi} & \wt Y_r^{\ge [f], *},
    }\] where $i_\xi$ and $\wt i_\xi$ are natural embeddings. Note that $\a \circ j = \wt i_\xi$. Thus \[((\b_f)_!\wt{\eta_f^* \CL_{\psi, r}})|_\xi \cong R\G_c(\wt i_\xi^*\wt{\eta_f^* \CL_{\psi, r}}) \cong R\G_c(j^* \a^*\wt{\eta_f^* \CL_{\psi, r}}) \cong R\G_c(j^* \eta_f^* \CL_{\psi, r}).\]

    Assume that $h\i g h = (x_{f'})_{f' \in \Psi}$. By definition $x_{f'} = 0$ if $[f'] < [f]$ and $x_{f'} \neq 0$ for some $f' \in [f]$. Then we have \[\eta_f \circ j: \b_f\i(\xi) \cong \BA^{D_f} \to \fkl, \quad (y_{f'})_{f' \in D_f} \mapsto \sum_{f' \in [f]} \a_{f'}^\vee(1 + \varpi^r c_{f'} x_{f'} y_{f'}),\] where $c_{f'} \in \CO_\brk^\times$ with $f' \in [f]$ are constants. Hence $\eta_f \circ j$ is a non-trivial linear map. Since $\a_{f'} \in \Phi \sm \Phi_L$ and $X_\psi$ is $(L, G)$-generic, we have $R\G_c((\eta_f \circ j)^* \CL_{\psi, r}) = 0$ as desired. The proof is finished.  
\end{proof}

\begin{proposition}\label{correspond}
    We have $\pInd_{P_{s, r}}^{G_r}(\CL_{\psi, r}[\dim \fkl]) \cong \CF_{\psi, r}[\dim \fkg]$.
\end{proposition}
\begin{proof}
    Let $\wh G_r^\fkl = \wt G_r \cap \eta\i(\fkl)$ and $\wt G_r^\fkl = \a(\wh G_r^\fkl)$. Then $\a\i(\wt G_r^\fkl) = \wh G_r^\fkl$. Since the support of  $\CL_{\psi, r}$ is contained in $\fkl$, we have \[\pInd_{P_{s, r}}^{G_r} \CL_{\psi, r} \cong (\pi|_{\wt G_r^\fkl})_! \wt{(\eta|_{\wt G_r^\fkl})^*\CL_{\psi, r}}.\] Note that \[\wt G_r^\fkl =\wt G_r^{\ge [r]} \sqcup \bigsqcup_{[r] > [f] \in \Psi / \sim_L} \wt G_r^{\ge [f],*}.\] Notice that $p\i(\fkl)^{\ge [r]} = \fkp$, and hence $\wt G_r^{\ge [r]}$ coincides with $\wt\fkg$ defined in \Cref{subsec:ind-Lie}. Thanks to \Cref{vanish} we have $(\pi|_{\wt G_r^{\ge [f], *}})_! \wt{(\eta|_{\wt G_r^{\ge [f], *}})^*\CL_{\psi, r}} = 0$ for $r > f \in \Psi$. Hence \[\pInd_{P_{s, r}}^{G_r} (\CL_{\psi, r}[\dim \fkl]) \cong (\pi|_{\wt G_r^{\ge [r]}})_! \wt{(\eta|_{\wt G_r^{\ge [r]}})^*\CL_{\psi, r}}[\dim \fkl] \cong {\mathfrak{pInd}}_{P_{s, r}}^{G_r} \CL_{\psi, r} [\dim \fkl] \cong \CF_\psi[\dim \fkg],\] where the second isomorphism follows from $\wt G_r^{\ge [r]} = \wt \fkg$, and the last one follows from \Cref{ind-Lie}. The proof is finished.
\end{proof}

\subsection{}
Let $T \subseteq P' = L' N'$ be anathor parabolic subgroup, where $T \subseteq L'$ is a Levi subgroup and $N' \subseteq P'$ is the unipotent radical. One can defined subgroups $P_{s, r}' = L_r' N_{s, r}'$ in a similar way. 

\begin{lemma} \label{dimension}
    Let notation be as above. Then \begin{itemize}
        \item $\dim N_{s, r} = \dim N_r$;

        \item $\dim L_r' = \dim L_r' \cap L_r + \dim L_r' \cap N_{s, r} + \dim L_r' \cap \ov N_{s, r}$;

        \item $\dim N_{s, r}' = \dim N_{s, r}' \cap L_r + \dim N_{s, r}' \cap N_{s, r} + \sharp \{f; f(\bx) = s, \a_f \in \Phi_{N' \cap \ov N}\}$.
    \end{itemize} Here $\ov N$ denotes the opposite of $N$.
\end{lemma}

\begin{corollary} \label{fiber-dim}
    We have \[\dim P_{s, r}'\cap N_{s, r} + \dim P_{s, r}' / (P_{s, r}' \cap P_{s, r}) = \dim L_r' \cap N_{s, r} + \dim N_{s, r}.\]
\end{corollary}
\begin{proof}
    Applying \Cref{dimension} we have \begin{align*}
        &\quad\ \dim P_{s, r}'\cap N_{s, r} + \dim P_{s, r}' / (P_{s, r}' \cap P_{s, r}) - \dim L_r' \cap N_{s, r} - \dim N_{s, r} \\
        &=\dim L_r' + \dim N_{s, r}' - \dim L_r' \cap L_r - \dim N_{s, r}' \cap L_r - \dim L_r' \cap N_{s, r} - \dim N_{s, r} \\
        &= \dim L_r' \cap \ov N_{s, r} + \dim N_{s, r}' - \dim N_{s, r}' \cap L_r - \dim N_{s, r} \\
        &= (\dim N_{s, r}' - \dim N_{s, r}' \cap L_r) - (\dim N_{s, r} - \dim L_r' \cap N_{s, r}) \\
        &= \sharp \{f; f(\bx) = s, \a_f \in \Phi_{N' \cap \ov N} \} - \sharp \{f; f(\bx) = s, \a_f \in \Phi_{N \cap \ov N'} \} \\ 
        &= 0,
    \end{align*} where $\ov N'$ denotes the opposite of $N'$. The proof is finished.
\end{proof}

Let $J_{s, r} = N_{s, r}[\fkl, L_r] \subseteq P_{s, r}$ and denote by $\bar p: P_{s, r} \to P_{s, r} / J_{s,r} = L_r / [\fkl, L_r]$ the natural quotient map.
\begin{lemma} \label{tech-1}
    Let $h \in G_r \sm P_{s, r}' N_{G_r}(T_r) P_{s, r}$. There exists $\a \in \Phi \sm \Phi_L$ and a group embedding $\z: \BG_a \hookrightarrow h\i N'_{s, r} h \cap P_{s, r}$ such that $(\bar p \circ \z) = \fkt^\a$. In particular, $R\G_c(\BG_a, (p \circ \z)^* \CF) = 0$ for any $\CF \in D_{L_r}^\psi(L_r)$.
\end{lemma}
\begin{proof}
    We adopt the notation in \Cref{subsec:aff-root}. Fix a Borel subgroup $T \subseteq B' \subseteq P'$, and let $\ge'$ be an associated linear order on $\tPhi$. Let \begin{align*} D' &= \{f \in \tPhi; f(\bx) = 0; \a_f \in \Phi_{N'} \sm w(\Phi_L)\}; \\ D &= \{f \in \tPhi; f(\bx) = 0, \a_f \in \Phi_N \sm w\i(\Phi_{L'})\} \end{align*} By replacing $h$ with some suitable element in $L_r' h L_r$ we may assume that $h \in u' w u G_{0+:r} $, where $w \in N_{G_r}(T_r)$, $u' \in G_r^{D'}$ and $u \in G_r^D$, see \Cref{subsec:aff-root}.

    First we assume that $u \neq 1$. Let $f \in D$ such that $\pr_f(u) \neq 0$ and $\pr_{f'}(u) = 0$ for all $f > f' \in D$. Note that $-\a_{w(f)} \in \Phi_{N'}$. Define \[\z: \BG_a \to G_r, \quad z \mapsto h\i u' u_{r - w(f)}(z) {u'}\i h.\] Then $u' G_r^{r-w(f)} {u'}\i \subseteq N'_{r:r} \subseteq N'_{s, r}$ and \[h\i u' G_r^{r-w(f)} {u'}\i h = u\i G_r^{r-f} u \equiv \fkt^{\a_f} \mod N_{s, r}.\] Hence $\z$ satisfies our requirements. 

    Second we assume $u = 1$. If $u' \neq 1$. Let $f \in D'$ such that $\pr_f(u') \neq 0$ and $\pr_{f'}(u') = 0$ for all $ f >' f' \in D'$. Define \[\z: \BG_a \to G_r, \quad z \mapsto h\i u_{r - f}(z) h.\] Then $G_r^{r-f} \subseteq N_{s, r}'$ and \[h\i G_r^{r-f} h = w\i {u'}\i G_r^{r-f} u' w \equiv \fkt^{\a_{w\i(f)}} \mod N_{s, r} + [\fkl, L_r].\] Hence $\z$ meets our requirements.

    Finally we assume $u = u' = 1$. Then $h \in w v$ with $v \in G_{0+:r}$. Let $\tPhi(P_{s, r})$ (resp. $\tPhi(P_{s, r}')$) be the set of affine roots in $\tPhi$ which appears in $P_{s, r}$ (resp. $P'_{s, r}$). Again we may assume that there exists $f \in \tPhi \sm (\tPhi(P_{s, r} ) \cup w\i(\tPhi(P_{s, r}'))$ such that $\pr_f(v) \neq 0$ and $\pr_{f'}(v) = 0$ for all $f > f' \in \tPhi \sm (\tPhi(P_{s, r} ) \cup w\i(\tPhi(P_{s, r}'))$. Define \[\z: \BG_a \to G_r, \quad z \mapsto h\i u_{r - w(f)}(z) h.\] Then $G_r^{r - w(f)} \in N'_{s, r}$ (since $w(f) \notin \tPhi(P_{s, r}')$) and \[h\i G_r^{r-w(f)} h = v\i G_r^{r-f} v \equiv \fkt_{w\i(\a_f)} \mod N_{s, r}.\] Hence $\z$ meets our requirements, and the proof is finished. 
\end{proof}

%Let $N_{G_r}(T_r)$ denote the normalizer of $T_r$ in $G_r$. Define $W_{G_r} = N_{G_r}(T_r)/T_r$, and define $W_{L_r}$ and $W_{L_r'}$ in the same way.
\begin{proposition}\label{Mackey}
    Let notation be as above. For $M \in D_{L_r}^\psi(L_r)$ we have \[\pRes_{P_{s, r}' !}^{G_r} \pInd_{P_{s, r} !}^{G_r} (M) \cong \bigoplus_{w \in W_{L_r'} \backslash W_{G_r} / W_{L_r}} \pInd_{L_r' \cap {}^w P_{s, r} !}^{L_r'} \pRes_{P_{s, r}' \cap {}^w L_r !}^{{}^w L_r}({}^w M),\] where ${}^w M$ denotes the pull-back of $M$ under $\ad(w)\i: {}^w L_r \to L_r$.
\end{proposition}
\begin{proof}
    We follow the idea of \cite[Proposition 5.9]{BC24}. Let \[Y = \{(g, h P_{s, r}) \in P_{s, r}' \times G_r / P_{s, r}; h\i g h \in P_{s, r}\} \subseteq \wt G_r.\] Set $\wh Y = \a\i(\wt Y)$, $\eta_Y = \eta|_{\wh Y}$, $\a_Y = \a|_{\wh Y}$ and $\pi_Y = \pi|_{\wt Y}$. Then we have \[\pRes_{P_{s, r}' !}^{G_r} \pInd_{P_{s, r} !}^{G_r} (M) \cong \pi_{Y!} \wt{\eta_Y^* M}[2 \dim N_{s, r}].\] Note that \[\wt Y = \wt Y' \sqcup \wt Y'',\] where $\wt Y'' = \pr_2\i(P'_{s, r} N_{G_r}(T_r) P_{s, r})$ and $\pr_2$ denote the projection to the second item. Let $\wh Y' = \a\i(\wt Y')$, $\wh Y'' = \a\i(\wt Y'')$, $\eta_Y' = \eta_Y|_{\wh Y'}$,  $\eta_Y'' = \eta_Y|_{\wh Y''}$ and $\pi_Y' = \pi_Y |_{\wt Y'}$ and $\pi_Y'' = \pi_Y |_{\wt Y''}$. The first step is to show that $\pi'_{Y!} \wt{(\eta_Y')^* M} = 0$. This follows from the arguments of loc. cit. by using \Cref{tech-1}. Then it remains to show \[\pi''_{Y!} \wt{(\eta_Y'')^* M}[2 \dim N_{s, r}] \cong \bigoplus_{w \in W_{L_r'} \backslash W_{G_r} / W_{L_r}} \pInd_{L_r' \cap {}^w P_{s, r} !}^{L_r'} \pRes_{P_{s, r}' \cap {}^w L_r !}^{{}^w L_r}({}^w M).\] This follows in the same way as in loc. cit. by using \Cref{fiber-dim}.
\end{proof}

\subsection{}
Let $\varphi: G_r \to G_r / N_{s, r}$ be the natural quotient map. Consider the following two functors \[\varphi_! \circ \mathrm{For}^{G_r}_{P_{s,r}}, \ \varphi_* \circ \mathrm{For}^{G_r}_{P_{s,r}}: D_{G_r}(G_r) \to D_{P_{s,r}}(G_r / N_{s, r}),\] which we denote by $\varphi_!$ and $\varphi_*$ for simplicity.

Similar as in \cite[\S 5.1]{BC24}, we have the following two general lemmas.
\begin{lemma}\label{phimonoidal}
    The functors $\varphi_!$ (resp. $\varphi_*$) is monoidal with respect to $\star_!$ (resp. $\star_*$).
\end{lemma}

\begin{lemma} \label{phi-star}
    Let $M \in D_{L_r}(L_r)$ and $N \in D_{G_r}(G_r)$ such that $M \star_!\varphi_!N$ is supported on $L_r=P_{s,r}/N_{s,r}$. we have \[\pInd_{P_{s, r} !}^{G_r}(M \star_! \varphi_! N) \cong \pInd_{P_{s, r} !}^{G_r}(M) \star_! N.\]
\end{lemma}

Let $J_{s, r} = N_{s, r}[\fkl, L_r] \subseteq P_{s, r}$. 
\begin{lemma} \label{tech-2}
    Let $h \in G_r \sm P_{s, r}$. Then there exists $\a \in \Phi \sm \Phi_L$ and $g \in L_r$ such that \[g \fkt^\a g\i h \subseteq \{J_{s, r} z h z\i J_{s, r}; z \in L_r \}.\]
\end{lemma}
\begin{proof}
   Write $h = u' w u \d$, where $w \in N_{G_r}(T_r)$, $u, u' \in U_r$ and $\d \in G_{0+:r}$. Note that there exists root subsystem $\Psi \subseteq \Phi(G_0, T_0)$ such that (the natural image of) $w$ is an elliptic element of the Weyl group of $\Psi$. In particular, $\fkt^\a \dot w \subseteq \{z \dot w z\i; z \in \fkt\}$ for all $\a \in \Psi$. 

   Assume $w \notin L_r$. Then there exists $\a \in \Psi \sm \Phi_L$. Let $z \in \fkt$ such that $z w z\i = x w$ with $x \in \fkt^\a$. Then \[z h z\i = [z, u'] u' z w z\i u [u\i, z] \d  \subseteq \fku u' x w u \d \fku = \fku x h \fku.\] Hence $x h \subseteq \fku z h z\i \fku$ and $\fkt^\a h \subseteq \{\fku z h z\i \fku; z \in \fkt\}$ as desired.

   Assume that $w \in L_r$. Let \[D = \{f \in \tPhi ; \text{ either } f(\bx) >  0 \text{ or } f(\bx) = 0 \text{ and } \a_f \in \Phi^+ \sm \Phi_L\}.\] Then we can write $h = m v$, where $m \in L_r$ and $v \in G_r^D$. Let $f \in D$ such that $\pr_f(u) \neq 0$ and $\pr_{f'}(u) = 0$ for all $f > f' \in D$. Note that $f \notin \tPhi(N_{s, r})$ since $h \notin P_{s, r}$. Then $G_r^{r - f} \in N_{s, r}$. Let $y \in G_r^{r - f}$. Then $[v, y] = J_{s, r} x$ with $x \in \fkt^{\a_f}$. Then \[ h y  = m y [y\i, v] v \subseteq N_{s, r} m J_{s, r} x v \subseteq J_{s, r} x h.\] Hence $x h \in J_{s, r} h N_{s, r}$ and $\fkt^{\a_f} h \in J_{s, r} \fkt^{\a_f} h N_{s,r}$ as desired.   
\end{proof}

\begin{proposition}\label{descent}
    Let $N \in D_{G_r}^\psi(G_r)$. Then \[\CL_{\psi, r} \star_! \varphi_! N \cong \CL_{\psi, r} \star_! \pRes_{P_{s, r} !}^{G_r} N \] is supported on $P_{s, r} / N_{s, r} \cong L_r$ and is direct summand of $\varphi_! N$. Similar statement also holds by replacing $\varphi_!$, $\star_!$ and $\pRes_{P_{s, r} !}^{G_r} $ with $\varphi_*$ and $\star_*$, $\pRes_{P_{s, r} *}^{G_r}$  respectively.   
\end{proposition}
\begin{proof}
    The proof follows the same argument as in \cite[Proposition 5.10]{BC24}. For completeness, we include the details here and clarify why Lemma~\ref{tech-2} is essential in our setting. Observe that $\varphi_!N$ and $\CL_{\psi, r}$ are $P_{s,r}$-equivariant under conjugation. Hence so is $\CL_{\psi, r} \star_! \varphi_! N$. Since $\CL_{\psi, r} \star_! \varphi_! N$ is $(\mathfrak{l},\CL_{\psi, r})$-equivariant and $X_{\psi}$ is $(L,G)$-generic, this implies $\CL_{\psi, r}$ is equivariant under the right multiplication action of $\CM := [L_r,\fkl]$. Consider the following map
    $$f: \CM \times P_{s,r}\times \CM \times G_r/N_{s,r}\to G_r/N_{s,r}$$
    $$(m,p,m',hN_{s,r})\mapsto (m p h p\i N_{s,r} m').$$
    The equivariant properties yield an isomorphism of sheaves:
    \begin{equation}\label{descenteq1}
    f^\ast(\CL_{\psi, r} \star_! \varphi_! N)|_{\CM \times  P_{s,r}\times \CM \times \{hN_{s,r}\}}\cong \delta_{\CM \times P_{s,r} \times \CM}\boxtimes \CL_{\psi, r} \star_! \varphi_! N|_{hN_{s,r}}
    \end{equation}
    For any $h\in G_r\sm P_{s, r}$. Lemma~\ref{tech-2} shows that \[f(\CM \times P_{s,r}\times \CM \times \{gN_{s,r}\} )=\{mnzhz^{-1}n'm';m,m'\in \CM ,n,n'\in N_{s,r},z\in L_r\}/N_{s,r}\] contains $g \fkt^\a g^{-1}h$ for some $\Phi \sm \Phi_L$ and $g\in L_r$. On the other hand, $\CL_{\psi, r} \star_! \varphi_! N$ is $(\fkl,\CL_{\psi, r})$-equivariant and the generic condition ensures that the restriction $\CL_{\psi, r}|_{g \fkt^\a g\i h}$ remains nontrivial. Combining this with the constancy condition from \Cref{descenteq1}, we necessarily obtain $\CL_{\psi, r} \star_! \varphi_! N|_{hN_{s,r}}=0$.

    Since $\CL_{\psi,r} \star_! \varphi_! N$ is supported on $P_{s,r}/N_{s,r} \cong L_r$, we obtain the natural identification $\mathcal{L}_{\psi,r} \star_! \varphi_! N \cong \mathcal{L}_{\psi,r} \star_! i^* \varphi_! N$ via the inclusion $i: L_r \hookrightarrow G_r/N_{s,r}$. On the other hand, by the proper base change theorem for the Cartesian square
    \begin{equation*}
  \xymatrix{
    P_{s,r} \ar[r] \ar[d] & G_r \ar[d]^{\varphi} \\
    L_r \ar[r]_-{i} & G_r/N_{s,r}
  }
\end{equation*}
  we have $\pRes_{P_{s,r}!}^{G_r}(N) \cong i^* \varphi_! N$. it follows that $\CL_{\psi,r} \star_! \pRes_{P_{s,r}!}^{G_r}(N) \cong \CL_{\psi,r} \star_! \varphi_! N$. 
\end{proof}

\begin{corollary}\label{resmonoidal}
The functor $\mathcal{L}_{\psi, r} \star_! \pRes_{P_{s, r}!}^{G_r}$ on $D_{G_r}^\psi(G_r)$ is monoidal with respect to the $\star_!$-convolution product. An analogous statement holds for $\ast$.
\end{corollary}
\begin{proof}
    By \Cref{phimonoidal}, $\varphi_!$ is monoidal hence so is $\CL_{\psi, r} \star_! \varphi_!=\CL_{\psi, r} \star_! \pRes_{P_{s, r} !}^{G_r}$
\end{proof}

%\begin{lemma}\label{Lsupport}Let $\phi:\mathfrak {g}\to \mathfrak {g}/\mathfrak{n}_{s,r}$. Then the pushforward $\phi_!\mathcal{F}_\psi$ is supported on $\mathfrak{p}_{s,r}/\mathfrak{n}_{s,r}$.\end{lemma}

\begin{lemma}\label{Mackey2}
    For $N \in D_{G_r}^\psi(G_r)$ we have \[\varphi_! N \cong \bigoplus_{w \in W_L \backslash W / W_L} \varphi_! N \star_! \pInd_{L_r \cap {}^w P_{s, r} !}^{L_r} \pRes_{P_{s, r} \cap {}^w L_r !}^{{}^w L_r}({}^w \CL_{\psi, r})  \]
\end{lemma}
\begin{proof}
The argument in \Cref{descent} establishes that for any $N \in D_{G_r}(G_r)$ such that $\varphi_! N$ is supported on $P_{s,r}/N_{s,r} \cong L_r$, we have $\varphi_! N \cong \pRes^{G_r}_{P_{s,r}!}(N)$. As $\CF_\psi$ is supported on $\fkg \subseteq P_{s, r}$, this implies that
\[
\varphi_! \mathcal{F}_{\psi} \cong \pRes^{G_r}_{P_{s,r}!}(\mathcal{F}_\psi).
\]
Applying \Cref{correspond} and \Cref{Mackey}, we deduce the decomposition:
\[
\varphi_! \mathcal{F}_{\psi} \cong \bigoplus_{w \in W_L \backslash W / W_L} \pInd_{L_r \cap {}^w P_{s, r} !}^{L_r} \pRes_{P_{s, r} \cap {}^w L_r !}^{{}^w L_r}({}^w \mathcal{L}_{\psi, r}).
\]

Now, for any $N \in D^\psi_{G_r}(G_r)$, the isomorphism $N \cong N \star_! \mathcal{F}_\psi$ holds. By the monoidal property in \Cref{phimonoidal}, we obtain:
\[
\varphi_!(N \star_! \mathcal{F}_\psi) \cong \varphi_! N \star_! \varphi_! \mathcal{F}_{\psi}.
\]
The desired result follows from this identification.
\end{proof}
\begin{theorem} \label{one-step-detail}
    The following statements hold:
    \begin{itemize}
        \item The two functors $\CL_{\psi, r} \star_! \pRes_{P_{s, r} !}^{G_r}$ and $\pInd_{P_{s, r} !}^{G_r}$ are inverse equivalences between $D_{L_r}^\psi(L_r)$ and $D_{G_r}^\psi(G_r)$, and similar for $\CL_{\psi, r} \star_* \pRes_{P_{s, r} *}^{G_r}$ and $\pInd_{P_{s, r} !}^{G_r}$;

        \item $\CL_{\psi, r} \star_! \pRes_{P_{s, r} !}^{G_r} \cong \CL_{\psi, r} \star_* \pRes_{P_{s, r} *}^{G_r}$ on $D_{G_r}^\psi(G_r)$ and $\pInd_{P_{s, r} !}^{G_r} \cong \pInd_{P_{s, r} *}^{G_r}$ on $D_{L_r}^\psi(L_r)$;

        \item $\pInd_{P_{s, r} !}^{G_r}$ is $t$-exact on $D_{L_r}^\psi(L_r)$. In particularly, $\pInd_{P_{s,r}!}^{G_r}$ sends simple perverse sheaves lying in $D_{L_r}^\psi(L_r)$ to simple perverse sheaves lying in $D_{G_r}^\psi(G_r)$.
    \end{itemize}
\end{theorem}
\begin{proof}
 The statement follows in the same way as \cite[Proposition 5.15]{BC24} by using \Cref{ind}, \Cref{correspond}, \Cref{Mackey}, \Cref{phi-star}, \Cref{descent}, \Cref{resmonoidal}, \Cref{Mackey2} and Artin's theorem.
\end{proof}

\bibliography{bib_ADLV}{}
\bibliographystyle{amsalpha}

\end{document}